\documentclass[12pt]{article}
\usepackage{amsmath}
\usepackage{amsmath,amssymb, amsthm, latexsym, amsfonts, epsfig, color,graphicx,}
\usepackage{mathrsfs}
\usepackage{indentfirst}
\usepackage{chngcntr}
\usepackage{bbm}
\usepackage{newtxmath} 
\usepackage[colorlinks=True,linkcolor=blue,anchorcolor=blue,citecolor=blue,CJKbookmarks=true]{hyperref}
\usepackage{geometry}
\begin{document}
\renewcommand{\a}{\alpha}
\newcommand{\D}{\Delta}
\newcommand{\ddt}{\frac{d}{dt}}
\counterwithin{equation}{section}
\newcommand{\e}{\epsilon}
\newcommand{\eps}{\varepsilon}
\newtheorem{theorem}{Theorem}[section]
\newtheorem{proposition}{Proposition}[section]
\newtheorem{lemma}[proposition]{Lemma}
\newtheorem{remark}{Remark}[section]
\newtheorem{example}{Example}[section]
\newtheorem{definition}{Definition}[section]
\newtheorem{corollary}{Corollary}[section]
\makeatletter
\newcommand{\rmnum}[1]{\romannumeral #1}
\newcommand{\Rmnum}[1]{\expandafter\@slowromancap\romannumeral #1@}
\makeatother

\title{Almost everywhere convergence of the convolution type Laguerre expansions}
\author{Longben Wei\\
{\small {\it  School of Mathematics and Statistics, Huazhong University of Science }}\\
{\small {\it and Technology, Wuhan, {\rm 430074,} P.R.China}}  \\
{\small {\it Email:longbenwei51@gmail.com}}\\
}
\date{}
\maketitle
\centerline{\bf Abstract }
 For a fixed d-tuple $\alpha=(\alpha_1,...,\alpha_d)\in(-1,\infty)^d$, consider the product space $\mathbb{R}_+^d:=(0,\infty)^d$ equipped with Euclidean distance $\arrowvert \cdot \arrowvert$ and the measure $d\mu_{\alpha}(x)=x_1^{2\alpha_1+1}\cdot\cdot\cdot x_{d}^{\alpha_d}dx_1\cdot\cdot\cdot dx_d$. We consider the Laguerre operator $L_{\alpha}=-\Delta+\sum_{i=1}^{d}\frac{2\alpha_j+1}{x_j}\frac{d}{dx_j}+\arrowvert x\arrowvert^2$ which is a compact, positive,  self-adjoint operator on $L^2(\mathbb{R}_+^d,d\mu_{\alpha}(x))$. In this paper, we study almost everywhere convergence of the Bochner-Riesz means associated with $L_\alpha$ which is defined by $S_R^{\lambda}(L_\alpha)f(x)=\sum_{n=0}^{\infty}(1-\frac{e_n}{R^2})_{+}^{\lambda}P_nf(x)$. Here $e_n$ is n-th eigenvalue of $L_{\alpha}$, and $P_nf(x)$ is the n-th Laguerre spectral projection operator. This corresponds to the convolution-type Laguerre expansions introduced in Thangavelu's lecture \cite{TS3}. The fact that the kernel of this type of Laguerre expansion is a radial function suggests that we can examine the sharpness of summability indices by radial functions. For $2\leq p<\infty$, we prove that 
 $$\lim_{R\rightarrow\infty} S_R^{\lambda}(L_\alpha)f=f\,\,\,\,-a.e.$$
 for all $f\in L^p(\mathbb{R}_+^d,d\mu_{\alpha}(x))$, provided that $\lambda>\lambda(\alpha,p)/2$, where $\lambda(\alpha,p)=\max\{2(\arrowvert\alpha\arrowvert_1+d)(1/2-1/p)-1/2,0\}$, and $\arrowvert\alpha\arrowvert_1:=\sum_{j=1}^{d}\alpha_{j}$.  Conversely, if $2\arrowvert\alpha\arrowvert_{1}+2d>1$, we will show the convergence generally fails if $\lambda<\lambda(\alpha,p)/2$ in the sense that there is an $f\in L^p(\mathbb{R}_+^d,d\mu_{\alpha}(x))$ for $(4\arrowvert\alpha\arrowvert_{1}+4d)/(2\arrowvert\alpha\arrowvert_{1}+2d-1)< p$ such that the convergence fails. When $2\arrowvert\alpha\arrowvert_{1}+2d\leq1$, our results show that a.e. convergence holds for $f\in L^p(\mathbb{R}_+^d,d\mu_{\alpha}(x))$ with $p\geq 2$ whenever $\lambda>0$.
 \newline 
\newline
\noindent{\bf 2020 Mathematics Subject Classification:} 42B25, 42B15, 42C10, 35P10.
\newline 
\noindent{\bf Key words:} Bochner-Riesz means, Almost everywhere convergence, Laguerre expansions, Laguerre operator.
\section{Introduction}
For $\lambda\geq0$ and $R>0$, the classical Bochner-Riesz means for the Laplacian on $\mathbb{R}^d$ are defined by
\begin{equation*}\label{Section1.1}
S_R^{\lambda}f(x)=\int_{\mathbb{R}^d}e^{2\pi ix\cdot\xi}\left(1-\frac{\arrowvert\xi\arrowvert^2}{R^2}\right)_{+}^{\lambda}\hat{f}(\xi)d\xi,\,\,\,\forall \xi\in\mathbb{R}^d.
\end{equation*}
Here $t_{+}=\max\{0,t\}$ for $t\in\mathbb{R}$ and $\hat{f}$ denotes the Fourier transform of $f$. The longstanding open problem known as the \emph{Bochner-Riesz conjecture} is that, for $1\leq p\leq\infty$ and $p\neq2$, $S_R^{\lambda}f\rightarrow\,f$ in $L^p(\mathbb{R}^d)$ norm if and only if 
\begin{equation}\label{Section1.2}
\lambda>\lambda(p):=\max\{d\arrowvert\frac{1}{2}-\frac{1}{p}\arrowvert-\frac{1}{2},0\}.
\end{equation}
It was shown by Herz \cite{Herz} that the condition \eqref{Section1.2} on \(\lambda\) is necessary for $S_R^{\lambda}f\rightarrow\,f$ in $L^p(\mathbb{R}^d)$ norm. In two dimensions, the conjecture has been proved by Carleson and Sj\"{o}lin \cite{Carleson}. In higher dimensions, only partial results are known, see \cite{Bourgain,Guth,Lee,Tao2, Tao3} and references therein. However, the conjecture still remains open for $d\geq3$. 

Concerning pointwise convergence, the problem of determining the optimal summability index $\lambda$ (depending on $p$) for which $S_R^{\lambda}f\rightarrow f$ almost everywhere for every $f\in L^p({\mathbb{R}^d})$ has also been extensively studied by various authors. In particular, for $2\leq p\leq\infty$, this problem was essentially settled by Carbery, Rubio de Francia and Vega \cite{Carbery1}. They showed that a.e. convergence holds for all $f\in L^p({\mathbb{R}^d})$ if $\lambda$ satisfies the condition \eqref{Section1.2} provided $p\geq2$. Discussion on the necessity of the condition \eqref{Section1.2} can be found in \cite{Lee2}. It should be mentioned that almost everywhere convergence of $S_R^{\lambda}f$ with $f\in L^p$, $1 < p < 2$, exhibits different nature and few results are known in this direction except when dimension $d=2$ (\cite{Li,Tao,Tao1}).

For a general positive self-adjoint operator which admits a spectral decomposition $L=\int_{0}^{\infty}t dE_{L}(t)$ in a $L^2$ space, where $dE_{L}$ denotes the spectral measure associated with $L$. The Bochner-Riesz means associated with $L$ are defined by
\begin{equation*}
S_{R}^{\lambda}(L)(f):=\int_{0}^{R^2}(1-\frac{t}{R^2})^{\lambda}dE_{L}(t)f.
\end{equation*}
The Bochner-Riesz means for elliptic operators (e.g. Schr\"{o}dinger operator with potentials, the Laplacian on manifolds) have attracted a lot of attention and have been studied extensively by many authors. See, for example, \cite{COSY, CS, DOS, GHS, LR, JLR, JR, KST, MYZ, S} and references therein.
\par In this paper we are concerned with Bochner-Riesz means associated with the Laguerre operator on $\mathbb{R}_{+}^{d}:=(0,\infty)^{d}$, which is given by 
\begin{equation}\label{Laguerre}
L_{\alpha}=-\Delta+\sum_{j=1}^{d}\frac{2\alpha_j+1}{x_j}\frac{d}{dx_j}+\arrowvert x\arrowvert^2,
\end{equation}
where $\alpha=(\alpha_1,\alpha_2,\cdot\cdot\cdot,\alpha_d)\in (-1,\infty)^d$. Laguerre operator is closely related to the twisted Laplacian (special Hermite operator) in the following sense (see \cite{SVK}): If $f\in \mathcal{S}(\mathbb{C}^{d})$ is radial, then $\mathcal{L}f(z)=\frac{1}{2}L_{d-1}f(\sqrt{2}t)$, where $\arrowvert z\arrowvert=\sqrt{2}t$, $L_{d-1}$ is the 1-dimensional Laguerre operator of type $d-1$, $\mathcal{L}$ is the twisted Laplacian on $\mathbb{C}^{d}$ given by 
\begin{equation*}
\mathcal{L}=-\sum_{j=1}^{d}\left(\left(\frac{\partial}{\partial x_j}-\frac{i}{2}y_j\right)^2+\left(\frac{\partial}{\partial y_j}+\frac{i}{2}x_j\right)^2\right),\quad (x,y)\in \mathbb{R}^{d}\times\mathbb{R}^{d}.
\end{equation*}
\subsection{Bochner-Riesz means associated with the Laguerre operator \(L_{\alpha}\)}
The one-dimensional Laguerre polynomial $L_n^{\alpha}(x)$ of type $\alpha>-1$ are defined by the generating function identity
$$\sum\limits_{n=0}^{\infty}L_n^{\alpha}(x)t^{k}=(1-t)^{-\alpha-1}e^{-\frac{xt}{1-t}},\,\,\,\arrowvert t\arrowvert<1.$$
Here $x>0$, $n\in \mathbb{Z}_{\geq 0}$ and $\mathbb{Z}$ is the standard symbol of the integer set. Each $L_n^{\alpha}(x)$ is a polynomial of degree $n$ explicity given by 
$$L_n^{\alpha}(x)=\sum\limits_{j=0}^{n}\frac{\Gamma(n+\alpha+1)}{\Gamma(n-j+1)\Gamma(j+\alpha+1)}\frac{(-x)^j}{j!}.$$
 It is well-known that the Laguerre  polynomials are orthonormal on $\mathbb{R}_{+}=(0,\infty)$ with respect to the measure $e^{-x}x^{\alpha}dx$. The Laguerre functions $\varphi_n^{\alpha}(x)=\left(\frac{2n!}{\Gamma(n+\alpha+1)}\right)^{1/2}L_n^{\alpha}(x^2)e^{-x^2/2}$ form a complete orthonormal family in $L^2(\mathbb{R}_+,x^{2\alpha+1}dx)$.  
It is well known that each $\varphi_n^{\alpha}(x)$ is an eigenfunction of 1-dimensional $L_{\alpha}$ with eigenvalue $(4n+2\alpha+2)$, i.e.,
\begin{equation*}
L_{\alpha}\varphi_n^{\alpha}(x)=(4n+2\alpha+2)\varphi_n^{\alpha}(x).
\end{equation*}
The operator $L_{\alpha}$ is non-negative and self-adjoint respect to the measure $x^{2\alpha+1}dx$ on $\mathbb{R}_+$.
\par Now for each multiindex $\mu=(\mu_1,...,\mu_d)\in \mathbb{Z}^{d}_{\geq0}$ and $\alpha=(\alpha_1,...,\alpha_d)\in (-1,\infty)^d$, the $d$-dimensional Laguerre functions are defined by the tensor product of the 1-dimensional Laguerre functions (see, e.g., \cite{SVK})
 \begin{equation*}
\varphi_{\mu}^{\alpha}(x)=\prod\limits_{j=1}^{d}\varphi_{\mu_j}^{\alpha_j}(x_j),\quad x\in \mathbb{R}_{+}^{d}.
\end{equation*}
 The corresponding $d$-dimensional Laguerre operator $L_{\alpha}$ for $\alpha=(\alpha_1,...,\alpha_d)\in (-1,\infty)^d$ is defined as the sum of 1-dimensional Laguerre operators $L_{\alpha_j}$ given by \eqref{Laguerre}.
It is straightforward to get that $L_{\alpha}\varphi_{\mu}^{\alpha}(x)=(4\arrowvert \mu\arrowvert_1 +2\arrowvert \alpha \arrowvert_1 +2d)\varphi_{\mu}^{\alpha}(x)$, where $\arrowvert \mu\arrowvert_1=\sum_{j=1}^{d}\mu_j$, $\arrowvert \alpha\arrowvert_1=\sum_{j=1}^{d}\alpha_j$. Hence, $\varphi_{\mu}^{\alpha}(x)$ are eigenfunctions of $L_{\alpha}$ with eigenvalue $4\arrowvert \mu\arrowvert_1 +2\arrowvert \alpha \arrowvert_1 +2d$ and they also form a complete orthonormal system in $L^{2}(\mathbb{R}_{+}^d,d\mu_{\alpha}(x))$. The d-dimensional Laguerre operator $L_{\alpha}$ is self-adjoint with respect to measure $d\mu_{\alpha}(x):=x_1^{2\alpha_1+1}\cdot\cdot\cdot x_{d}^{2\alpha_d+1}dx_1\cdot\cdot\cdot dx_d$. Throughout this paper, for $1\leq p<\infty$, we denote by \( L^p(\mathbb{R}_{+}^d,d\mu_{\alpha}(x)) \) the Banach space consisting of \( \mu_{\alpha} \)-measurable functions on \( \mathbb{R}_{+}^d \), equipped with the norm
\begin{equation}\label{Norm}
\lVert f \rVert_{p} := \left( \int_{\mathbb{R}_{+}^d} \lvert f(x) \rvert^p \, d\mu_{\alpha}(x) \right)^{1/p}<\infty.
\end{equation}
\par For function \(f\in L^{2}(\mathbb{R}_{+}^d,d\mu_{\alpha}(x))\), it has the  Laguerre expansion
\begin{equation}\label{exp1}
f(x)=\sum_{\mu\in\mathbb{Z}_{\geq 0}}\langle f, \varphi_{\mu}^{\alpha}\rangle_{\alpha}\varphi_{\mu}^{\alpha}(x)=\sum_{n=0}^{\infty}P_{n}f(x),
\end{equation}
where $\langle\cdot,\cdot \rangle_{\alpha}$ is the inner product inherited from $L^{2}(\mathbb{R}_{+}^d,d\mu_{\alpha}(x))$ and 
\begin{equation}\label{Projection}
P_{n}f(x):=\sum_{\arrowvert\mu\arrowvert_1=n}\langle f, \varphi_{\mu}^{\alpha}\rangle_{\alpha}\varphi_{\mu}^{\alpha}(x).
\end{equation}
We note that expansion \eqref{exp1} is called Laguerre expansions of convolution type, introduced in Thangavelu's lecture \cite{TS3}. An interesting point about this expansion is that it is in a weighted \( L^2 \) space. Regarding this type of expansion, the existing literature primarily focuses on the study of Ces\`{a}ro means for this orthogonal system. For $R>0$, the $R>0$-th Ces\`{a}ro means of order $\lambda$ are defined by 
\begin{equation*}
C_{R}(\lambda)f(x):=\frac{1}{A_{R^2}(\lambda)}\sum_{n\leq R^2}A_{R^2-n}P_nf(x),
\end{equation*}
where $A_k(\lambda)=\Gamma(k+\lambda+1)/\Gamma(k+1)\Gamma(\lambda+1)$. 
\par For $R>0$, the $R>0$-th Bochner-Riesz means for the Laguerre operator $L_{\alpha}$ of order $\lambda\geq0$  is defined by 
\begin{equation}\label{Bochner}
S_{R}^{\lambda}(L_{\alpha})f(x)=\sum_{n=0}^{\infty}\left(1-\frac{e_n}{R^2}\right)_{+}^{\lambda}P_nf(x),
\end{equation}
where $e_n=4n+2\arrowvert \alpha\arrowvert_1+2d$ is the n-th eigenvalue of $L_{\alpha}$. The assumption $\lambda\geq0$ is necessary for $S_{R}^{\lambda}(L)$ to be defined for all $R>0$, see \cite{CD}. To proceed with our statement, we define
\begin{equation}
\lambda(\alpha,p):=\max\{2(\arrowvert\alpha\arrowvert_1+d)\arrowvert1/2-1/p\arrowvert-1/2,0\}.
\end{equation}
\par According to a theorem by Gergen \cite{Gn}, as pointed out by Thangavelu in \cite{TS}, the Bochner-Riesz means can be expressed in terms of the Cesàro means and vice versa. Moreover,  the $L^p$ norm convergence of one type of mean implies that of the other. Thus, when concerning the $L^p$ convergence of $S_{R}^{\lambda}(L_{\alpha})f$, in one dimension, it is shown in \cite{TS3} that $S_{R}^{\lambda}(L_{\alpha})f$ converges in $L^{p}(\mathbb{R}_{+}, x^{2\alpha+1}dx)$ for $1\leq p\leq \infty$, whenever $\alpha\geq 0$ and $\lambda> \alpha+\frac{1}{2}$. For $\alpha\geq 0$ and $0<\lambda\leq \alpha+\frac{1}{2}$, $S_{R}^{\lambda}(L_{\alpha})f$ converges in $L^{p}(\mathbb{R}_{+}, x^{2\alpha+1}dx)$ for $1< p<\infty$ if and only if $\lambda>\lambda(\alpha,p)$. In higher dimensions ($d\geq2$), the $L^p$ convergence of $S_{R}^{\lambda}(L_{\alpha})f$ is not as well understood. For $\alpha=(\alpha_1,...,\alpha_d)\in (-1,\infty)^d$ and $1\leq p\leq\infty$, it is shown by Xu in \cite {XY} that if $\alpha_i\geq0$, for all $1\leq i\leq d$, then $S_{R}^{\lambda}(L_{\alpha})f$ converges in $L^{p}(\mathbb{R}_{+}, x^{2\alpha+1}dx)$ whenever $\lambda>\arrowvert \alpha\arrowvert_1+d-1/2$. Moreover, for $p=1$ and $\infty$, $S_{R}^{\lambda}(L_{\alpha})$ converge in $L^{p}(\mathbb{R}_{+}^d, d\mu_{\alpha}(x))$ if and only if $\lambda> \arrowvert \alpha\arrowvert_1+d-1/2=\lambda(\alpha,p)$. For $\lambda\leq \arrowvert \alpha\arrowvert_1+d-1/2$ and $1< p<\infty$, it is also conjectured in \cite{XY} that $S_{R}^{\lambda}(L_{\alpha})f$ converges in $L^{p}(\mathbb{R}_{+}, x^{2\alpha+1}dx)$ if and only if $\lambda>\lambda(\alpha,p)$. Xu also proved that $S_{R}^{\lambda}(L_{\alpha})f$ converges in $L^{p}(\mathbb{R}_{+}, x^{2\alpha+1}dx)$ if and only if $\lambda>\lambda(\alpha,p)$ under the assumption that $f$ is radial, thus the condition $\lambda>\lambda(\alpha,p)$ is necessary for $L^p$ convergence of $S_{R}^{\lambda}(L_{\alpha})f$.
\subsection{Almost everywhere convergence}
Motivated by the recent work of Chen-Duong-He-Lee-Yan \cite{CD} on a.e. convergence of Bochner-Riesz means for the Hermite operator $H:=-\Delta +\arrowvert x\arrowvert^2$, this paper is concerned with the almost everywhere convergence of $S_{R}^{\lambda}(L_{\alpha})f$. Specifically, we aim to characterize the values of $\lambda$ for which 
\begin{equation}\label{Almost}
\lim_{t\rightarrow\infty}S_{R}^{\lambda}(L_{\alpha})f(x)=f(x) \quad a.e.\quad\quad\forall f\in L^p(\mathbb{R}_{+}^{d},d\mu_{\alpha}(x))
\end{equation}
for $2\leq p<\infty$, where ``a.e."
 refers to convergence almost everywhere with respect to the measure $d\mu_{\alpha}(x)$. Compared with $L^p$ convergence of $S_{R}^{\lambda}(L_{\alpha})f$, its a.e. convergence has not been well studied, and the existing literature on this topic is quite limited. In one dimension, when $\alpha=k-1 \in \mathbb{Z}_{\geq 0}$, Thangavelu showed in \cite{TS1} that \eqref{Almost} holds for $2\leq p\leq\infty$ if $\lambda>(2k-1)(1/2-1/p)$, where these summability indices exceed $\lambda(\alpha,p)=2k(1/2-1/p)$. This result was based on his study of the almost everywhere convergence of Bochner-Riesz means for the twisted Laplacian, specifically when restricted to radial functions. Recent advancements in the almost everywhere convergence of Bochner-Riesz means for the twisted Laplacian \cite{JLR} imply that, for $d=1$ and $\alpha=k-1 \in \mathbb{Z}_{\geq 0}$ , then \eqref{Almost} holds for $2\leq p\leq\infty$ if $\lambda>\lambda(\alpha,p)/2$, thereby extending  Thangavelu's finding \cite{TS1} to a broader range. However, for general $\alpha>-1$ and in high-dimensional settings, we have not found any relevant research on this topic.
\par The following is the main result of this paper, which almost completely settles the a.e. convergence problem of $S_{R}^{\lambda}(L_{\alpha})f$ except the endpoint cases. Here, $\alpha \in(-1,\infty)^d$ is general and the dimension $d$ is arbitrary.
\begin{theorem}\label{T1.1}
Let $\alpha \in(-1,\infty)^d$, $2\leq p<\infty$, and $\lambda>0$. Then, for any $f\in L^p(\mathbb{R}_{+}^{d},d\mu_{\alpha}(x))$, we have \eqref{Almost} whenever $\lambda>\lambda(\alpha,p)/2$. In particular, for $2(\arrowvert \alpha\arrowvert_1+d)\leq 1$, \eqref{Almost} holds for all $f\in L^p(\mathbb{R}_{+}^{d},d\mu_{\alpha}(x))$ whenever $\lambda>0$. Conversely, if \eqref{Almost} holds for all $f\in L^p(\mathbb{R}_{+}^{d},d\mu_{\alpha}(x))$ with $2(\arrowvert \alpha\arrowvert_1+d)> 1$ and $(4\arrowvert\alpha\arrowvert_{1}+4d)/(2\arrowvert\alpha\arrowvert_{1}+2d-1)<p<\infty$, we have $\lambda\geq\lambda(\alpha,p)/2$.
\end{theorem}
From our Theorem \ref{T1.1}, we observe that the critical summability index for almost everywhere convergence is half of that required for $L^p$ convergence. A similar result was established in \cite{CD,JLR} for the Bochner-Riesz means associated with the Hermite operator and the twisted Laplacian. This improvement in the summability index is related to the fact that these operators have discrete spectra that are bounded away from zero, which allows the inhomogeneous weight $(1+\arrowvert x\arrowvert )^{-\beta}$ used in the following Theorem \ref{T1.2} to be effective. Additionally, we note that when $\alpha=(-\frac{1}{2}, -\frac{1}{2},\cdot\cdot\cdot, -\frac{1}{2})$, the Laguerre operator $L_{\alpha}$ reduces to Hermite operator on $\mathbb{R}_{+}^{d}$, and in this case, the summability indices in Theorem \ref{T1.1} align with those studied in \cite{CD}.
\par The sufficiency part of Theorem \ref{T1.1} relies on the maximal estimate, a standard tool in the study of almost everywhere convergence. To prove \eqref{Almost}, we introduce the corresponding maximal operator \( S_{*}^{\lambda}(L_{\alpha}) \), defined by
\[
S_{*}^{\lambda}(L_{\alpha})f(x) := \sup_{R>0} \left| S_{R}^{\lambda}(L_{\alpha})f(x) \right|.
\]
We will establish the following weighted estimate, from which the almost everywhere convergence of \( S_{R}^{\lambda}(L_{\alpha})f \) follows through standard arguments.
\begin{theorem}\label{T1.2}
Let $\alpha \in(-1,\infty)^d$ and $0\leq\beta<2(\arrowvert \alpha\arrowvert_1 +d)$. The maximal operator $S_{*}^{\lambda}(L_{\alpha})$ is bounded on $L^2(\mathbb{R}_{+}^{d},(1+\arrowvert x\arrowvert )^{-\beta}d\mu_{\alpha}(x))$ if 
\begin{equation}
\lambda>\max\{\frac{\beta-1}{4},0\}.
\end{equation}
Conversely, if $S_{*}^{\lambda}(L_{\alpha})$ is bounded on $L^2(\mathbb{R}_{+}^{d},(1+\arrowvert x\arrowvert ) ^{-\beta}d\mu_{\alpha}(x))$, then $\lambda\geq\max\{\frac{\beta-1}{4},0\}$.
\end{theorem}
Now, we can deduce the sufficiency part of Theorem \ref{T1.1} from Theorem \ref{T1.2}. Indeed, Theorem \ref{T1.2} establishes a.e. convergence of $S_{R}^{\lambda}(L_{\alpha})f$ for all $f\in L^2(\mathbb{R}_{+}^{d},(1+\arrowvert x\arrowvert ) ^{-\beta}d\mu_{\alpha}(x))$ provided that $\lambda>\max\{\frac{\beta-1}{4},0\}$ via a standard argument. For given $p\geq 2$ and $\lambda>\lambda(\alpha,p)/2$, we can choose a $\beta$ such that $\beta>(2\arrowvert \alpha\arrowvert_1+2d)(1-2/p)$ and $\lambda>\max\{\frac{\beta-1}{4},0\}$. This choice of $\beta$ ensures that  $f\in L^2(\mathbb{R}_{+}^{d},(1+\arrowvert x\arrowvert) ^{-\beta}d\mu_{\alpha}(x))$ if $f\in L^2(\mathbb{R}_{+}^{d},d\mu_{\alpha}(x))$ as it follows by H\"{o}lder's inequality. Therefore, this yields a.e. convergence of $S_{R}^{\lambda}(L_{\alpha})f$ for $f\in L^2(\mathbb{R}_{+}^{d},d\mu_{\alpha}(x))$ if $\lambda>\lambda(\alpha,p)/2$. 
\par The use of weighted \( L^2 \) estimates in the study of pointwise convergence for Bochner–Riesz means dates back to Carbery et al. \cite{Carbery1}. It was subsequently found that this strategy is equally effective for similar problems in other contexts, as demonstrated in works such as \cite{Annoni, Lee2, JLR}. The approach adopted in this paper is not new and is primarily based on the argument presented in \cite{CD}
\subsection{Organization}
In Section \ref{Section2}, we establish several preparatory estimates that are crucial for the proof of Theorem \ref{T1.2}. Section \ref{Section3} is dedicated to proving the sufficiency part of Theorem \ref{T1.2}, which is reduced to the square function estimate, the proof of which is deferred to appendix \ref{appendix}. In Section \ref{Section4}, we demonstrate the sharpness of the summability indices, thereby completing the proofs of Theorem \ref{T1.1} and Theorem \ref{T1.2}.
\subsection{Notation}
 For $x\in \mathbb{R}_{+}^{d} $ and $r>0$, we denote by $B(x,r)=\{y\in\mathbb{R}_{+}^{d}: \arrowvert x-y\arrowvert<r\}$ the ball centered at x with radius r.  We emphasize that the notations $\arrowvert x\arrowvert_1$ and $\arrowvert x\arrowvert$ are distinct, where $\arrowvert x\arrowvert_1:=\sum_{i=1}^{d}x_i$ and $\arrowvert x\arrowvert:=\left(\sum_{i=1}^{d}x_i^2\right)^{1/2}$. For a subset $E\subset\mathbb{R}_{+}^{d}$, the notation $\arrowvert E\arrowvert$ denotes the Lebesgue measure of $E$. For any bounded function $\mathscr{M}$ the operator $\mathscr{M}(L_{\alpha})$ is defined by $\mathscr{M}(L_{\alpha}):=\sum_{n=0}^{\infty}\mathscr{M}(4n+2\arrowvert \alpha\arrowvert_1+2d)P_n$.  We note that the coupling constant $\alpha\in (-1,\infty)^{d}$ is considered fixed throughout, and will not be explicitly mentioned thereafter. In particular, the uniform constants in the subsequent inequalities may depend on \( \alpha \) and the dimension \( d \), unless otherwise stated. 
\section{Preparatory estimates}\label{Section2}
In this section, we will establish several preparatory estimates that are crucial for the proof of Theorem \ref{T1.2}. 
\subsection{kernel estimate}
In this subsection, we derive the expression for the heat-diffusion kernel of the operator $L_{\alpha}$ and show that it satisfies Gaussian upper bounds. These bounds, in turn, imply the finite speed propagation property of the wave operator $\cos t\sqrt{L_{\alpha}}$. The finite speed propagation property will be instrumental in proving the high-frequency part of the square function estimate (see Proposition \ref{SP} below). The d-dimensional heat-diffusion kernel of the operator $L_{\alpha}$ is defined by (see, for example, \cite{SVK})
\begin{equation*}
K_{L_{\alpha}}(t,x,y)=\sum_{n=0}^{\infty}e^{-t(4n+2\arrowvert\alpha\arrowvert_1+2d)}\sum_{\arrowvert\mu\arrowvert=n}\varphi_{\mu}^{\alpha}(x)\varphi_{\mu}^{\alpha}(y),
\end{equation*}
where $\varphi_{\mu}^{\alpha}$ are the eigenfunctions of the operator $L_{\alpha}$ corresponding to eigenvalue $e_n=4n+2\arrowvert\alpha\arrowvert_1+2d$.
\begin{lemma} The kernel $K_{L_{\alpha}}(t,x,y)$ has the following explicit form:
\begin{equation}\label{HeatK}
K_{L_{\alpha}}(t,x,y)=\left(\frac{2e^{-2t}}{1-e^{-4t}}\right)^d\exp\left(-\frac{1}{2}\frac{1+e^{-4t}}{1-e^{-4t}}(\arrowvert x\arrowvert^2+\arrowvert y\arrowvert^2)\right)\prod_{i=1}^{d}\frac{I_{\alpha_i}(\frac{2e^{-2t}x_iy_i}{1-e^{-4t}})}{(x_iy_i)^{\alpha_i}},
\end{equation}
where $I_{\alpha}(z)=i^{-\alpha}J_{\alpha}(z)$ is the modified Bessel function ($J_{\alpha}(z)$ being the standand Bessel function of the first kind).
\end{lemma}
\begin{proof}
We repeat the proof of \cite[Lemma 3.2]{SVK} here, as the setting of our Laguerre operators $L_{\alpha}$ is slightly different, though not in essence. We begin by expanding the kernel $K_{L_{\alpha}}(t,x,y)$ using its definition:
\begin{equation*}
\begin{aligned}
K_{L_{\alpha}}(t,x,y)&=\sum_{n=0}^{\infty}e^{-t(4n+2\arrowvert\alpha\arrowvert_1+2d)}\sum_{\arrowvert\mu\arrowvert=n}\varphi_{\mu}^{\alpha}(x)\varphi_{\mu}^{\alpha}(y)\\
&=\sum_{n=0}^{\infty}e^{-t(4n+2\arrowvert\alpha\arrowvert_1+2d)}\sum_{\arrowvert\mu\arrowvert=n}\prod_{i=1}^{d}\varphi_{\mu_i}^{\alpha_i}(x_i)\varphi_{\mu_i}^{\alpha_i}(y_i)\\
&=e^{-t(2\arrowvert\alpha\arrowvert_1+2d)}\prod_{i=1}^{d}\sum_{\mu_i=0}^{\infty}e^{-4t\mu_i}\varphi_{\mu_i}^{\alpha_i}(x_i)\varphi_{\mu_i}^{\alpha_i}(y_i)\\
&=\left(\frac{2e^{-2t}}{1-e^{-4t}}\right)^d\exp\left(-\frac{1}{2}\frac{1+e^{-4t}}{1-e^{-4t}}(\arrowvert x\arrowvert^2+\arrowvert y\arrowvert^2)\right)\prod_{i=1}^{d}\frac{I_{\alpha_i}(\frac{2e^{-2t}x_iy_i}{1-e^{-4t}})}{(x_iy_i)^{\alpha_i}}.
\end{aligned}
\end{equation*}
In the third equality, we use the fact that both functions are in $L^2(\mathbb{R}_{+}^{2d},d\mu_{\alpha}(x,y))$ and share the same Fourier–Laguerre coefficients. The final equality follows from the following Mehler's formula for Laguerre polynomials (see, for example,  \cite[1.1.47]{TS3}):
\begin{equation*}
\sum_{k=0}^{\infty}\frac{\Gamma(k+1)}{\Gamma(k+\alpha+1)}L_{k}^{\alpha}(x)L_{k}^{\alpha}(y)z^k=(1-z)^{-1}(xyz)^{-\frac{\alpha}{2}}e^{-\frac{z}{1-z}(x+y)}I_{\alpha}\left(\frac{2\sqrt{xyz}}{1-z}\right), \quad \forall\arrowvert z\arrowvert<1.
\end{equation*}
\end{proof}
\begin{lemma}
There exist $C,c>0$ such that
\begin{equation}\label{kernel}
K_{L_{\alpha}}(t,x,y)\leq\frac{C}{e^{dt}\mu_{\alpha}(Q(x,\sqrt{t})}\exp(-c\frac{\arrowvert x-y\arrowvert^2}{t})
\end{equation}
for all $x,y\in \mathbb{R}_{+}^d$ and $t>0$, where $Q(x,\sqrt{t})$ is defined by tensor product of the one-dimension balls, that is,  $Q(x,\sqrt{t}):=B(x_1,\sqrt{t})\times\cdot\cdot\cdot \times B(x_n,\sqrt{t}).$
\end{lemma}
\begin{proof}
By \eqref{HeatK}, we have 
\begin{equation*}
K_{L_{\alpha}}(t,x,y)=\prod_{i=1}^{d}\frac{2e^{-2t}}{1-e^{-4t}}\exp\left(-\frac{1}{2}\frac{1+e^{-4t}}{1-e^{-4t}}( x_i^2+ y_j^2)\right)\frac{I_{\alpha_i}(\frac{2e^{-2t}x_iy_i}{1-e^{-4t}})}{(x_iy_i)^{\alpha_i}}.
\end{equation*}
To prove the inequality \eqref{kernel}, it suffices to establish the following estimates for each component i:
\begin{equation}\label{K1}
\frac{2e^{-2t}}{1-e^{-4t}}\exp\left(-\frac{1}{2}\frac{1+e^{-4t}}{1-e^{-4t}}( x_i^2+ y_j^2)\right)\frac{I_{\alpha_i}(\frac{2e^{-2t}x_iy_i}{1-e^{-4t}})}{(x_iy_i)^{\alpha_i}}\leq \frac{C_i}{e^t\mu_{\alpha_i}(B(x_i,\sqrt{t}))}\exp(-c_i\frac{\arrowvert x_i-y_i\arrowvert^2}{t})
\end{equation}
for some constants $C_i,c_i>0$. Here $\mu_{\alpha_i}(B(x_i,\sqrt{t}))$
is defined as $$\mu_{\alpha_i}(B(x_i,\sqrt{t}):=\int_{B(x_i,\sqrt{t})}x^{2\alpha_i+1}dx.$$
By Lemma 3.1 in \cite{BD}, the inequality \eqref{K1} holds true  with some constant $C_i,c_i>0$. Thus, the proof of the inequality \eqref{kernel} is complete.
\end{proof}
The Gaussian upper bounds imply the following finite-speed propagation property of the wave operator $\cos t\sqrt{L_{\alpha}}$.
\begin{lemma}\label{Finit}
There exists a constant $c_0>0$ such that
\begin{equation}\label{Finit.1}
supp\,K_{\cos(t\sqrt{L_{\alpha}})}(x,y)\subset D(t):=\{(x,y)\in \mathbb{R}_{+}^{d}\times\mathbb{R}_{+}^{d}:\arrowvert x-y\arrowvert\leq c_0t\},\quad \forall t>0,
\end{equation}
where $K_{\cos(t\sqrt{L_{\alpha}})}(x,y)$ is the kernel associated to $\cos(t\sqrt{L_{\alpha}})$.
\end{lemma}
For the proof of Lemma \ref{Finit}, we refer to \cite{HLMML,SA} for example.
\subsection{Muckenhoupt weights and Littlewwod-Paley inequality} It is straightforward to prove that there exists a constant $C>0$ so that 
\begin{equation*}
\mu_{\alpha}(B(x,\lambda r))\leq C \lambda^{2\arrowvert \alpha\arrowvert_1 +2d}\mu_{\alpha}(B(x, r))
\end{equation*}
for all $x\in \mathbb{R}_{+}^{d} $, $r>0$ and $\lambda>1$ (see Lemma \ref{M1} below for a detailed discussion). Thus $(\mathbb{R}_{+}^{d},d\mu_{\alpha}(x), \arrowvert \cdot\arrowvert)$ forms a space of homogeneous type, and we can define the class of Muckenhoupt weights as in \cite{BD}. In what follows, by a weight $\omega$, we shall mean that $\omega$ is a non-negative measurable and locally integrable function on $(\mathbb{R}_{+}^{d},d\mu_{\alpha}(x), \arrowvert \cdot\arrowvert)$. We say that a weight $\omega\in A_p(\mu_{\alpha})$, $1<p<\infty$, if there exists a constant $C$ such that for all balls $B\subset\mathbb{R}_{+}^{d}$,
\begin{equation*}\label{Mweights}
\left(\frac{1}{\mu_{\alpha}(B)}\int_{B}\omega(x)d\mu_{\alpha}(x)\right)\left(\frac{1}{\mu_{\alpha}(B)}\int_{B}\omega^{-1/(p-1)}(x)d\mu_{\alpha}(x)\right)^{p-1}\leq C.
\end{equation*}
The uncentered Hardy–Littlewood maximal operator on  $\mathbb{R}_{+}^{d}$ over balls is defined by
\begin{equation}\label{HLM}
\mathscr{M}(f)(x)=\sup_{x\in B}\frac{1}{\mu_{\alpha}(B)}\int_{B}\arrowvert f(y)\arrowvert d\mu_{\alpha}(y),\quad x\in \mathbb{R}_{+}^{d},
\end{equation}
where the supremum is taken over all balls $B\subset\mathbb{R}_{+}^{d}$ containing $x$. The theory of Hardy–Littlewood maximal operator and $A_p$ weights tells us that $\mathscr{M}$ is bounded on $L^p(\mathbb{R}_{+}^{d},d\mu_{\alpha}(x))$ for $1<p<\infty$ (see, for example, \cite[Proposition 2.14]{MS} ). Furthermore, if $\omega\in A_p(\mu_{\alpha})$, then $\mathscr{M}$ is also bounded on $L^p(\mathbb{R}_{+}^{d},\omega d\mu_{\alpha}(x))$. Next, we investigate for which real number $\beta$ the function $(1+\arrowvert x\arrowvert)^{\beta}$ is an $A_p(\mu_{\alpha})$ weight on $(\mathbb{R}_{+}^{d},d\mu_{\alpha}(x), \arrowvert \cdot\arrowvert)$. To proceed, we will use the following results, whose proofs are based on the strategies from \cite[Example 7.1.6, 7.1.7]{GL}.
\begin{lemma}\label{M1}
For every $\alpha=(\alpha_1,\cdot\cdot\cdot, \alpha_d)\in(-1,\infty)^d$, if $\beta>-2(\arrowvert \alpha\arrowvert_1+d)$, then the measure $\arrowvert x\arrowvert^{\beta}d\mu_{\alpha}(x)=\arrowvert x\arrowvert^{\beta}x_1^{2\alpha_1+1}\cdot\cdot\cdot x_d^{2\alpha_d+1}dx_1\cdot\cdot\cdot dx_d$ is a doubling measure on $\mathbb{R}_{+}^{d}$. Specifically, there exists a constant $C$ such that 
\begin{equation}\label{Doubling}
\int_{B(x_0,2R)}\arrowvert x\arrowvert^{\beta}d\mu_{\alpha}(x)\leq C\int_{B(x_0,R)}\arrowvert x\arrowvert^{\beta}d\mu_{\alpha}(x)
\end{equation}
for all balls $B\subset\mathbb{R}_{+}^{d}$.
\end{lemma}
\begin{proof}
 We divide all balls $B(x_0, R)$ in $\mathbb{R}_{+}^{d}$ into two categories: balls of type I that satisfy $\arrowvert x_0\arrowvert\geq 3R$ and type II that $\arrowvert x_0\arrowvert< 3R$. If $B(x_0, R)$ is of type $I$, then for $x\in B(x_0, R)$ we must have 
\begin{equation*}
\arrowvert x_0\arrowvert-R\leq \arrowvert x\arrowvert\leq \arrowvert x_0\arrowvert+R.
\end{equation*}
Thus for balls of type I, we observe that
\begin{equation}\label{M1.1}
\int_{B(x_0,2R)}\arrowvert x\arrowvert^{\beta}d\mu_{\alpha}(x)\leq \mu_{\alpha}(B(x_0,2R))\left\{\begin{array}{lr}(\arrowvert x_0\arrowvert +2R)^{\beta}\quad if \quad\beta\geq 0,\\
(\arrowvert x_0\arrowvert -2R)^{\beta}\quad if \quad\beta< 0,
\end{array}
\right.
\end{equation}
\begin{equation}\label{M1.2}
\int_{B(x_0,R)}\arrowvert x\arrowvert^{\beta}d\mu_{\alpha}(x)\geq \mu_{\alpha}(B(x_0,R))\left\{\begin{array}{lr}(\arrowvert x_0\arrowvert -R)^{\beta}\quad if \quad\beta\geq 0,\\
(\arrowvert x_0\arrowvert +R)^{\beta}\quad if \quad\beta< 0.
\end{array}
\right.
\end{equation}
Next, we claim that there exists a constant $C$, depending only on $\alpha$ and $d$ such that
\begin{equation}\label{MClaim}
\mu_{\alpha}(B(x,\lambda R))\leq C\lambda^{2\arrowvert \alpha\arrowvert_1+2d}\mu_{\alpha}(B(x,R))
\end{equation}
for all $x\in \mathbb{R}_{+}^{d} $ and $R>0$ and $\lambda>1$. This claim will allow us to conclude the proof of \eqref{Doubling} for balls of type I. Specifically, for $\arrowvert x_0\arrowvert\geq 3R$, we have $\arrowvert x_0\arrowvert+2R\leq 4(\arrowvert x_0\arrowvert-R)$ and $\arrowvert x_0\arrowvert-2R\geq \frac{1}{4}(\arrowvert x_0\arrowvert+R)$. Using these relations, we can apply the estimates from \eqref{M1.1}, \eqref{M1.2}, and \eqref{MClaim} to show that there exists a constant $C$ such that the doubling condition \eqref{Doubling} holds for balls of type I.
\par To prove \eqref{MClaim}, we consider the circumscribed cuboid \( Q_{out}(\lambda R)\) of ball $B(x_0,\lambda R)$ and the inscribed cuboid \( Q_{in}(R) \) of ball $B(x_0,R)$. These cuboids are defined as follows:
\begin{equation*}
\begin{aligned}
& Q_{out}(\lambda R)=[x_{01}-\lambda R,x_1+\lambda R]\times\cdot\cdot\cdot[x_{0d}-\lambda R,x_{0d}+\lambda R]\cap \mathbb{R}_{+}^{d} ,\\
& Q_{in}(R)=[x_{01}-R/\sqrt{d},x_1+R/\sqrt{d}]\times\cdot\cdot\cdot[x_{0d}-R/\sqrt{d},x_{0d}+R/\sqrt{d}]\cap \mathbb{R}_{+}^{d}.
\end{aligned}
\end{equation*}
 Since the measure space $(\mathbb{R}_{+}, x^{\alpha}dx, \arrowvert \cdot\arrowvert)$ is a space of homogeneous type for every $\alpha>-1$ (see, \cite[(1)]{BD}), there exists a constant $C>0 $, depending only on $\alpha$, such that
\begin{equation*}
\int_{B(x,\lambda R)} t^{\alpha}dt\leq C \lambda^{\alpha+1}\int_{B(x,R)} t^{\alpha}dt
\end{equation*}
for all $x\in \mathbb{R}_{+}$, $R>0$ and $\lambda>1$. By applying Fubini's theorem, we get
\begin{equation*}\label{M1.3}
\begin{aligned}
\mu_{\alpha}(Q_{out}(\lambda R))&=\int_{B(x_{01},\lambda R)}x_1^{2\alpha_1+1}dx_1\cdot\cdot\cdot \int_{B(x_{0d},\lambda R)}x_d^{2\alpha_d+1}dx_d\\
&\leq C d^{\arrowvert\alpha\arrowvert_1+d}\lambda^{2\arrowvert\alpha\arrowvert_1+2d}\int_{B(x_{01},R/\sqrt{d})}x_1^{2\alpha_1+1}dx_1\cdot\cdot\cdot \int_{B(x_{0d},R/\sqrt{d})}x_d^{2\alpha_d+1}dx_d\\
&\leq C\lambda^{2\arrowvert\alpha\arrowvert_1+2d}\mu_{\alpha}(Q_{in}(R)).
\end{aligned}
\end{equation*}
Now since $Q_{in}(R)\subset B(x_0,R)\subset B(x_0,\lambda R)\subset Q_{out}(\lambda R),$ we obtain
\begin{equation*}
\mu_{\alpha}(B(x_0,\lambda R))\leq \mu_{\alpha}(Q_{out}(\lambda R))\leq C\lambda^{2\arrowvert \alpha\arrowvert_{1} +2d}\mu_{\alpha}(Q_{in}(R))\leq C\lambda^{2\arrowvert \alpha\arrowvert_{1} +2d}\mu_{\alpha}(B(x_0, R)),
\end{equation*}
which is the desired estimate \eqref{MClaim}. 
\par For balls of type II, where $\arrowvert x_0\arrowvert\leq3R$,  we note two things: first, we have
\begin{equation*}
\int_{B(x_0,2R)}\arrowvert x\arrowvert^{\beta}d\mu_{\alpha}(x)\leq\int_{B(0,5R)}\arrowvert x\arrowvert^{\beta}d\mu_{\alpha}(x)\leq C_0R^{\beta+2\arrowvert \alpha_1\arrowvert+2d},
\end{equation*}
and second ,since $\arrowvert x\arrowvert^{\beta}$ is radially decreasing for $\beta<0$ and radially increasing for $\beta\geq0$, we have
\begin{equation*}
\int_{B(x_0,R)}\arrowvert x\arrowvert^{\beta}d\mu_{\alpha}(x)\geq\left\{\begin{array}{lr}\int_{B(0,R)} \arrowvert x\arrowvert^{\beta}d\mu_{\alpha}(x)\geq C_1R^{\beta+2\arrowvert\alpha\arrowvert_1+2d}\quad when\quad\beta\geq 0,\\
\int_{Q_{in}(R)}  (4R)^{\beta}d\mu_{\alpha}(x)\geq C_2R^{\beta+2\arrowvert\alpha\arrowvert_1+2d}\quad when\quad\beta< 0,
\end{array}
\right.
\end{equation*}
where constants $C_0$, $C_1$ and $C_2$ depends only on $\alpha$, $\beta$ and the dimension $d$. This establishes \eqref{Doubling} for balls of type II. Thus, the proof of the lemma is complete.
\end{proof}

Next, we investigate for which real number $\beta$ the power function $\arrowvert x\arrowvert^{\beta}$ is an $A_p(\mu_{\alpha})$ weight on $\mathbb{R}_{+}^{d}$. 
\begin{lemma}\label{M2}
For $1<p<\infty$ and a real number $\beta$, the power function $\arrowvert x\arrowvert^{\beta}$ is an $A_p(\mu_{\alpha})$ weight if and only if $-2(\arrowvert \alpha\arrowvert_1+d)<\beta<2(\arrowvert \alpha\arrowvert_1+d)(p-1)$.
\end{lemma}
\begin{proof}
For $1<p<\infty$, we examine for which value of $\beta$ the following expression is finite:
\begin{equation}\label{power}
\sup_{B\,cubes}\left(\frac{1}{\mu_{\alpha}(B)}\int_{B}\arrowvert x\arrowvert^{\beta}d\mu_{\alpha}(x)\right)\left(\frac{1}{\mu_{\alpha}(B)}\int_{B}\arrowvert x\arrowvert^{-\beta\frac{p'}{p}}d\mu_{\alpha}(x)\right)^{\frac{p}{p'}}.
\end{equation}
We divide all balls $B(x_0,R)$ in $\mathbb{R}_{+}^{d}$ into two categories as in the proof of the previous lemma: balls of type I that satisfy $\arrowvert x_0\arrowvert\geq 3R$ and type II that $\arrowvert x_0\arrowvert< 3R$.  If $B(x_0,R)$ is of type I, then for $x\in B(x_0,R)$ we must have
\begin{equation*}
\frac{2}{3}\arrowvert x_0\arrowvert\leq \arrowvert x_0\arrowvert-R\leq \arrowvert x\arrowvert\leq \arrowvert x_0\arrowvert+R\leq\frac{4}{3}\arrowvert x_0\arrowvert,
\end{equation*}
thus the expression inside the supremum in \eqref{power} is comparable to
\begin{equation*}
\arrowvert x_0\arrowvert^{\beta}\left(\arrowvert x_0\arrowvert^{-\beta\frac{p'}{p}}\right)^{\frac{p}{p'}}=1.
\end{equation*}
If $B(x_0,R)$ is of type II, then $B(0,5R)$ has size comparable to $B(x_0,R)$ and contains it. Since by Lemma \ref{M1} the measure $\arrowvert x\arrowvert^{\beta}d\mu_{\alpha}$ is doubling, the $\arrowvert x\arrowvert^{\beta}d\mu_{\alpha}$-measure of $B(x_0,R)$ and $B(0,5R)$ are comparable. It suffices therefore to estimate the expression inside the supremum in \eqref{power}, in which we have replaced $B(x_0,R)$ by $B(0,5R)$. Denote 
\begin{equation*}
C^*:=2(\arrowvert \alpha\arrowvert_1+d)\mu_{\alpha}(B(0,1)),
\end{equation*}
then 
\begin{equation*}
\mu_{\alpha}(B(0,5R))=\frac{C^{*}\cdot (5R)^{2(\arrowvert \alpha\arrowvert_1+d)}}{2(\arrowvert \alpha\arrowvert_1+d)}.
\end{equation*}
Thus, we obtain
\begin{equation*}
\begin{aligned}
&\left(\frac{1}{\mu_{\alpha}(B(0,5R))}\int_{B(0,5R)}\arrowvert x\arrowvert^{\beta} d\mu_{\alpha}(x)\right)\left(\frac{1}{\mu_{\alpha}(B(0,5R))}\int_{B(0,5R)}\arrowvert x\arrowvert^{-\beta\frac{p'}{p}}d\mu_{\alpha}(x)\right)^{\frac{p}{p'}}\\
&=\left(\frac{2(\arrowvert \alpha\arrowvert_1+d)C^{*}}{C^{*}\cdot (5R)^{2(\arrowvert \alpha\arrowvert_1+d)}}\int_{0}^{5R}r^{\beta+2\arrowvert \alpha\arrowvert_1+2d-1}dr\right)\left(\frac{2(\arrowvert \alpha\arrowvert_1+d)C^{*}}{C^{*}\cdot (5R)^{2(\arrowvert \alpha\arrowvert_1+d)}}\int_{0}^{5R}r^{-\beta\frac{p'}{p}+2\arrowvert \alpha\arrowvert_1+2d-1}dr\right)^{\frac{p}{p'}}.
\end{aligned}
\end{equation*}
This expression is finite and independent of $R$  if and only if $-2(\arrowvert\alpha\arrowvert_1+d)<\beta<\frac{p}{p'}2(\arrowvert\alpha\arrowvert_1+d)$.
\end{proof}
Based Lemma \ref{M2}, we conclude that $1+\arrowvert x\arrowvert^{\beta}\in A_2(\mu_{\alpha})$ and so is $(1+\arrowvert x\arrowvert)^{\beta}$ for $-2(\arrowvert \alpha\arrowvert_1+d)<\beta<2(\arrowvert \alpha\arrowvert_1+d)$ (see also \cite[p.16]{CD}). We can now deduce the following Littlewood-Paley inequality for the Laguerre operator $L_{\alpha}$.
\begin{proposition}\label{PL}
Fix a non-zero $C^{\infty}$ bump function $\psi$ on $\mathbb{R}$ such that supp\,$\psi\subseteq(\frac{1}{2},2)$. Let $\psi_k(t)=\psi(2^{-k}t)$, $k\in\mathbb{Z}$, for $t>0$. Then, for any $-2(\arrowvert \alpha\arrowvert_1+d)<\beta<2(\arrowvert \alpha\arrowvert_1+d)$
\begin{equation}\label{PL1.1}
\left\lVert \left(\sum_{k=-\infty}^{\infty}\arrowvert \psi_k(\sqrt{L_{\alpha}})f\arrowvert^2\right)^{1/2}\right\rVert_{L^2(\mathbb{R}_{+}^d,(1+\arrowvert x\arrowvert)^{\beta}d\mu_{\alpha}(x))}\leq C\lVert  f\rVert_{L^2(\mathbb{R}_{+}^d,(1+\arrowvert x\arrowvert)^{\beta}d\mu_{\alpha}(x))}.
\end{equation}
\end{proposition}
We omit the proof here since the argument is somewhat standard. One can see \cite[p1.6]{CD} for a sketch of the proof.

\subsection{Weighted estimates for the Laguerre operators $L_{\alpha}$}
In the spirit of \cite{BD,KEL}, we introduce the space of text functions as follows. Let $x=(x_1,x_2,\cdot\cdot\cdot,x_d)\in \mathbb{R}_{+}^{d}$, and let $m=(m_1,m_2,\cdot\cdot\cdot,m_d)$ and $k=(k_1,k_2,\cdot\cdot\cdot,k_d)\in \mathbb{Z}_{\geq 0}^{d}$. The notation $x^{m}$ denotes the monomial $x_1^{m_1}x_{2}^{m_2}\cdot\cdot\cdot x_d^{m_d}$, and differential operator $(x^{-1}D_x)^{k}$ is defined by $$\prod_{i=1}^{d}\left(x_i^{-1}\frac{\partial}{\partial x_{i}}\right)^{k_i}.$$
 Without any confusion, we abuse the notation $\mathscr{S}(\mathbb{R}_{+}^{d})$ to indicate our new space of text functions. The space of text functions $\mathscr{S}(\mathbb{R}_{+}^{d})$ is defined as a set of all function $\phi\in C^{\infty}(\mathbb{R}_{+}^{d})$ satisfying the following condition:
\begin{equation*}
\lVert \phi\rVert_{m,k}=\sup_{x\in \mathbb{R}_{+}^{d}}\arrowvert x^{m}(x^{-1}D_x)^{k}\phi(x)\arrowvert<\infty,\quad for\quad all \quad m, k\in \mathbb{Z}_{\geq 0}^{d},
\end{equation*}
with the topology on $\mathscr{S}(\mathbb{R}_{+}^{d})$ defined by the semi-norms $\lVert \cdot\rVert_{m,k}$. It is evident that the eigenfunctions of $L_{\alpha}$ belong to this function space. The main goal in this subsection is to establish the following weighted estimate for the Laguerre operators $L_{\alpha}$, which will be crucial for proving the low-frequency part of the square function estimate (Proposition \ref{SP} below).
\begin{lemma}\label{weighted}
Let $\beta\geq 0$. Then, the estimate 
\begin{equation}\label{weighted1.1}
\lVert (1+\arrowvert x\arrowvert)^{2\beta}f\rVert_{L^2(\mathbb{R}_{+}^{d},d\mu_{\alpha}(x)}\leq C\lVert (1+L_{\alpha})^{\beta}f\rVert_{L^2(\mathbb{R}_{+}^{d},d\mu_{\alpha}(x)}
\end{equation}
holds for any  $f\in \mathscr{S}(\mathbb{R}_{+}^{d})$.
\end{lemma}
To prove this lemma, we first need to establish several auxiliary estimates. It is important to note that the norm notation $\lVert\cdot\rVert_2$, which appears next, has already been clear by \eqref{Norm}.
\begin{lemma}\label{L3.1}
For all $\phi\in \mathscr{S}(\mathbb{R}_{+}^{d})$, we have $2(\arrowvert\alpha\arrowvert_1+d)\lVert\phi\rVert_{2}\leq \lVert L_{\alpha}\phi\rVert_{2}$ and  $2^m(\arrowvert\alpha\arrowvert_1+d)^m\lVert (L_{\alpha})^k\phi\rVert_{2}\leq\lVert (L_{\alpha})^{k+m}\phi\rVert_{2}$ for every $k,m\in \mathbb{N}$.
\end{lemma}
\begin{proof}
This follows from the fact that the first eigenvalue of $L_{\alpha}$ is bigger than or equal to $2(\arrowvert\alpha\arrowvert_1+d)$, as shown by the spectral properties of the operator $L_{\alpha}$.
\end{proof}
\begin{lemma}\label{L3.2}
Let $d=1$ and $\alpha>-1$. Then, for all $\phi\in \mathscr{S}(\mathbb{R}_{+})$,  we have the following estimate:
\begin{equation*}\label{L3.8}
\lVert x^2\phi\rVert_{2}^{2}+\left\lVert \frac{d^2}{dx^2}\phi\right\rVert_{2}^{2}+(2\alpha+1)^2\left\lVert\frac{1}{x}\frac{d}{dx}\phi\right\rVert_{2}^{2}+2\left\lVert x\frac{d}{dx}\phi\right\rVert_{2}^{2}\leq 3\lVert L_{\alpha} \phi\rVert_{2}^{2}.
\end{equation*}
\end{lemma}
\begin{proof}
Since $\lVert L_{\alpha}\phi\rVert_{2}^2=\langle(-\frac{d^2}{dx^2}+\frac{2\alpha+1}{x}\frac{d}{dx}+x^2)\phi,(-\frac{d^2}{dx^2}+\frac{2\alpha+1}{x}\frac{d}{dx}+x^2)\phi\rangle_{\alpha}$, a straightforward computation yields the following expansion:
\begin{equation*}\label{L3.9}
\begin{aligned}
& \lVert L_{\alpha}\phi\rVert_{2}^2=\left\lVert \frac{d^2}{dx^2}\phi\right\rVert_{2}^2+2Re\langle-\frac{d^2}{dx^2}\phi,\frac{2\alpha+1}{x}\frac{d}{dx}\phi\rangle_{\alpha}+2Re\langle-\frac{d^2}{dx^2}\phi,x^2\phi\rangle_{\alpha}\\
&+2(2\alpha+1)Re\langle\frac{1}{x}\frac{d}{dx}\phi,x^2\phi\rangle_{\alpha}+\left\lVert\frac{2\alpha+1}{x}\frac{d}{dx}\phi\right\rVert_{2}^2+\lVert x^2\phi\rVert_{2}^2.
\end{aligned}
\end{equation*}
Now, observe the following identities:
\begin{equation*}\label{L3.10}
2Re\langle-\frac{d^2}{dx^2}\phi,\frac{2\alpha+1}{x}\frac{d}{dx}\phi\rangle_{\alpha}=2\alpha(2\alpha+1)\left\lVert\frac{1}{x}\frac{d}{dx}\phi\right\rVert_{2}^2,
\end{equation*}
\begin{equation*} \label{L3.11}
2Re\langle-\frac{d^2}{dx^2}\phi,x^2\phi\rangle_{\alpha}=2\left\lVert x\frac{d}{dx}\phi\right\rVert_{2}^2+2(2\alpha+3)Re\langle\frac{1}{x}\frac{d}{dx}\phi,x^2\phi\rangle_{\alpha},
\end{equation*}
and 
\begin{equation*}\label{L3.12}
2Re\langle\frac{1}{x}\frac{d}{dx}\phi,x^2\phi\rangle_{\alpha}=-2(\alpha+1)\lVert \phi\rVert_{2}^2.
\end{equation*}
Using these relations, we deduce:
\begin{equation*}
\begin{aligned}
&\lVert x^2\phi\rVert_{2}^2+(2\alpha+1)^2\left\lVert\frac{1}{x}\frac{d}{dx}\phi\right\rVert_{2}^2+\left\lVert \frac{d^2}{dx^2}\phi\right\rVert_{2}^2+2\left\lVert x\frac{d}{dx}\phi\right\rVert_{2}^2\\
&=\lVert L_{\alpha}\phi\rVert_{2}^2+8(\alpha+1)^2\lVert \phi\rVert_{2}^2.
\end{aligned}
\end{equation*}
Finally, applying the estimate from Lemma \ref{L3.1}:
\begin{equation*}
2(\alpha+1)\lVert\phi\rVert_{2}\leq \lVert L_{\alpha}\phi\rVert_{2},
\end{equation*}
we arrive at the desired inequality:
\begin{equation*}
\lVert x^2\phi\rVert_{2}^2+(2\alpha+1)^2\left\lVert\frac{1}{x}\frac{d}{dx}\phi\right\rVert_{2}^2+\left\lVert \frac{d^2}{dx^2}\phi\right\rVert_{2}^2+2\left\lVert x\frac{d}{dx}\phi\right\rVert_{2}^2\leq 3\lVert L_{\alpha}\phi\rVert_{2}^2.
\end{equation*}
Thus, the proof is complete.
\end{proof}
\begin{lemma}\label{L4}
Let $d=1$. Then for $k\in \mathbb{N}$, there exist constants $C_k$, $D_k>0$ such that for all  $\phi\in \mathscr{S}(\mathbb{R}_{+})$
\begin{equation}\label{L3.13}
\lVert x^{2k}\phi\rVert_{2}\leq C_k \lVert L_{\alpha}^k\phi\rVert_{2}\quad and \quad \lVert L_{\alpha} (x^{2(k-1)}\phi)\rVert_{2}\leq D_k \lVert L_{\alpha}^k\phi\rVert_{2}.
\end{equation}
\end{lemma}
\begin{proof}
By Lemma \ref{L3.2}, we see that the estimate \eqref{L3.13} holds with $C_1=\sqrt{3}$ and $D_1=1$ if $k=1$. We now proceed by induction to prove the inequality for $k\geq 2$. Assume that the estimate \eqref{L3.13} holds for $k-1$ with some constants $C_{k-1}$ and $D_{k-1}$. Using the following identity for $ L_{\alpha} (x^{2(k-1)}\phi)$:
\begin{equation*}
 L_{\alpha} (x^{2(k-1)}\phi)= 4(k-1)(k-\alpha-2)x^{2(k-2)}\phi-4(k-1)x\frac{d}{dx}(x^{2(k-2)}\phi)+x^{2(k-1)}L_{\alpha}\phi.
\end{equation*}
we obtain the following estimate:
\begin{equation*}
\lVert L_{\alpha} (x^{2(k-1)}\phi)\rVert_{2}\leq 4(k-1)\arrowvert(k-\alpha-2)\arrowvert\lVert x^{2(k-2)}\phi\rVert_{2}+4(k-1)\lVert x\frac{d}{dx}(x^{2(k-2)}\phi)\rVert_{2}+\lVert x^{2(k-1)}L_{\alpha}\phi\rVert_{2}.
\end{equation*}
Applying the induction hypothesis, we get:
\begin{equation*}
\begin{aligned}
 & \lVert L_{\alpha} (x^{2(k-1)}\phi)\rVert_{2}\\
 &\leq 4(k-1)\arrowvert(k-\alpha-2)\arrowvert C_{k-2}\lVert L_{\alpha}^{k-2}\phi\rVert_{2}+4(k-1)C_1\lVert L_{\alpha}(x^{2(k-2)}\phi)\rVert_{2}+C_{k-1}\lVert (L_{\alpha})^{k-2}\phi\rVert_{2}\\
 &\leq 4(k-1)\arrowvert(k-\alpha-2)\arrowvert C_{k-2}\lVert L_{\alpha}^{k-2}\phi\rVert_{2}+4(k-1)C_1D_{k-1}\lVert L_{\alpha}^{k-1}\phi\rVert_{2}+C_{k-1}\lVert L_{\alpha}^{k-2}\phi\rVert_{2}.
 \end{aligned}
\end{equation*}
Combining this with Lemma \ref{L3.1} we obtain
\begin{equation*}
\lVert L_{\alpha} (x^{2(k-1)}\phi)\rVert_{2}\leq D_k\lVert L_{\alpha}^k\phi\rVert_{2}
\end{equation*}
with 
$$D_k=4(k-1)\arrowvert(k-\alpha-2)\arrowvert C_{k-2}2^{-2}(\alpha+1)^{-2}+4(k-1)C_1D_{k-1}2^{-1}(\alpha+1)^{-1}+C_{k-1}2^{-2}(\alpha+1)^{-2}.$$
On the other hand, by Lemma \ref{L3.2}, we also have
\begin{equation*}
\lVert x^{2k}\phi\rVert_{2}=\lVert x^2x^{2(k-1)}\phi\rVert_{2}\leq C_1\lVert  L_{\alpha} (x^{2(k-1)}\phi)\rVert_{2}\leq C_k\lVert L_{\alpha}^k\phi\rVert_{2}
\end{equation*}
with $C_k=C_1D_k$, $k\geq2$. This completes the proof. 
\end{proof}
Now, the weighted estimate \eqref{weighted1.1} we aim to prove, is a conclusion of above Lemma \ref{L4} and the L\"{o}wner-Heinz inequality (see, e.g., \cite[Section I.5]{CH}). We omit its proof here and refer the reader to \cite[p.10]{CD} for a detailed discussion.
\subsection{Trace lemma for the Laguerre operators}
In this subsection, we establish the following trace lemma, which will be used to prove the square function estimate (Proposition \ref{SP} below).
\begin{lemma}\label{Trace}
For $\beta>1$, there exists a constant $C>0$ such that the estimate
\begin{equation}\label{Trace1.1}
\lVert P_nf\rVert_{L^2(\mathbb{R}_{+}^{d},d\mu_{\alpha}(x))\rightarrow L^2(\mathbb{R}_{+}^{d}, (1+\arrowvert x\arrowvert)^{-\beta}d\mu_{\alpha}(x))}\leq Cn^{-\frac{1}{4}}
\end{equation}
holds for every $n\in \mathbb{N}$. 
\end{lemma}
Let $\mathscr{L}_{n}^{a}$ denote the normalized Laguerre function of type $a$ which is defined by
\begin{equation}\label{Define}
\mathscr{L}_{n}^{\alpha}(x)=\left(\frac{\Gamma(n+1)}{\Gamma(n+\alpha+1)}\right)^{\frac{1}{2}}e^{-\frac{x}{2}}x^{\frac{\alpha}{2}}L_{n}^{\alpha}(x).
\end{equation}
Then functions $\{\mathscr{L}_{n}^{\alpha}(x)\}$ form an orthonormal family in $L^2([0,\infty),dx)$.
To prove the trace lemma, we use the following result.
\begin{lemma}\label{L}\cite[Theorem 1.5.3]{TS3}. Let $\nu=4n+2a+2$ and $a>-1$.
\begin{equation}\label{L1.1}
\arrowvert\mathscr{L}_{n}^{a}(x)\arrowvert\leq C\left\{\begin{array}{lr}
(x\nu)^{a/2},\qquad\qquad\qquad\quad\quad\qquad\,\,\,\, 0\leq x\leq 1/\nu,\\
(x\nu)^{-1/4}, \qquad\qquad\quad\qquad\quad\qquad\,\, 1/\nu\leq x\leq \nu/2,\\
\nu^{-1/4}(\nu^{1/3}+\arrowvert\nu-x\arrowvert)^{-1/4},\qquad\quad\quad\nu/2\leq x\leq 3\nu/2,\\
e^{-\gamma x},\qquad\qquad\qquad\qquad\qquad\qquad\, x\geq 3\nu/2,
\end{array}
\right.
\end{equation}
where $\gamma>0$ is a constant. Moreover, if $1\leq x\leq\nu-\nu^{1/3}$, we have 
\begin{equation}\label{L1.2}
\mathscr{L}_{n}^{a}(x)=(2/\pi)^{\frac{1}{2}}(-1)^{n}x^{-\frac{1}{4}}(\nu-x)^{-\frac{1}{4}}\cos\left(\frac{\nu(2\theta-\sin 2\theta)-\pi}{4}\right)+O\left(\frac{\nu^{\frac{1}{4}}}{(\nu-x)^{\frac{7}{4}}}+(x\nu)^{-\frac{3}{4}}\right),
\end{equation}
where $\theta=\arccos(x^{1/2}\nu^{-1/2})$.
\end{lemma}
Now, we proceed to prove Lemma \ref{Trace}. The approach follows a similar argument to the one used in \cite{CD} for proving the trace lemma for the Hermite operator.
\begin{proof}
We begin with an local $L^{2}$ estimate. It suffices to prove equation \eqref{Trace1.1} by showing the following estimate:
\begin{equation}\label{Proof1.1}
\int_{[0,M]^d}\arrowvert P_nf(x)\arrowvert^2d\mu_{\alpha}(x)\leq CMn^{-\frac{1}{2}}\lVert f\rVert_{L^2(\mathbb{R}_{+}^{d},d\mu_{\alpha}(x))}^2
\end{equation}
for every $M\geq1$. Once this is established, equation \eqref{Trace1.1} follows immediately by decomposing $\mathbb{R}_{+}^{d}$ into dyadic shells and applying 
\eqref{Proof1.1} to each of them, since $\beta>1$.
To prove \eqref{Proof1.1}, we decomposes the Laguerre expansion of $f$ as follows:
\begin{equation*}\label{Proof1.2}
f=\sum_{i=1}^{d}f_i(x),
\end{equation*}
where the functions $f_1$,$\cdot\cdot\cdot$,$f_d$ are orthogonal to each other, and for each $1\leq i\leq d$, $\mu_i\geq \frac{\arrowvert\mu\arrowvert_1}{d}$ whenever $\langle f_i,\varphi_\mu^{\alpha}\rangle\neq0$ (see \cite{CD}). Note that $\mu_i\sim\arrowvert\mu\arrowvert_1$, so in order to show \eqref{Proof1.1}, it suffices to show
\begin{equation}\label{Proof1.3}
\int_{[0,M]^d}\arrowvert P_nf_i(x)\arrowvert^2d\mu_{\alpha}(x)\leq CMk^{-\frac{1}{2}}\sum_{\arrowvert \mu\arrowvert_1=n}\mu_i^{-\frac{1}{2}}\arrowvert\langle f_i,\varphi_\mu^{\alpha}\rangle_{\alpha}\arrowvert^2
\end{equation}
for each $i=1,...,d$ and $M>0$. By symmetry, it is enough to prove \eqref{Proof1.3} for $i=1$. Setting $c(\mu)=\langle f_i,\varphi_\mu^{\alpha}\rangle_{\alpha}$, we obtain
\begin{equation*}
\arrowvert P_nf_i(x)\arrowvert^2=\sum_{\arrowvert\mu\arrowvert_1=n}\sum_{\arrowvert\nu\arrowvert_1=n}c(\mu)\overline{c(\nu)}\prod_{i=1}^{d}\varphi_{\mu_i}^{\alpha_i}(x_i)\varphi_{\nu_i}^{\alpha_i}(x_i).
\end{equation*}
Using this, by Fubini's theorem it follows that
\begin{equation*}
\int_{[0,M]^d}\arrowvert P_nf_i(x)\arrowvert^2d\mu_{\alpha}(x)\leq\sum_{\arrowvert\mu\arrowvert_1=n}\sum_{\arrowvert\nu\arrowvert_1=n}c(\mu)\overline{c(\nu)}\int_0^{M}\varphi_{\mu_1}^{\alpha_1}(x_1)\varphi_{\nu_1}^{\alpha_1}(x_1)d\mu_{\alpha_i}(x_1)\prod_{i=2}^{d}\langle\varphi_{\mu_i}^{\alpha_i}(x_i),\varphi_{\nu_i}^{\alpha_i}(x_i)\rangle_{\alpha_i}.
\end{equation*}
Since $\varphi_{\mu_i}^{\alpha_i}$ are orthogonal to each other in $L^{2}([0,\infty),d\mu_{\alpha_i})$. We have $\mu_i=\nu_i$ for $i=2,\cdot\cdot\cdot,d$ whenever $\langle\varphi_{\mu_i}^{\alpha_i}(x_i),\varphi_{\nu_i}^{\alpha_i}(x_i)\rangle\neq0$ and this yields $\mu_1=\nu_1$ since $\arrowvert\mu\arrowvert_1=\arrowvert\nu\arrowvert_1=n$. Thus we obtain
\begin{equation*}
\int_{[0,M]^d}\arrowvert P_nf_i(x)\arrowvert^2d\mu_{\alpha}(x)\leq\sum_{\arrowvert\mu\arrowvert_1=n}\arrowvert c(\mu)\arrowvert^2\int_0^{M}\varphi_{\mu_1}^{\alpha_1}(x_1)\varphi_{\nu_1}^{\alpha_1}(x_1)d\mu_{\alpha_1}(x_1).
\end{equation*}
Now, to get the desire eatimate, it suffices to show that
\begin{equation*}
\int_0^{M}\left(\varphi_{\mu_1}^{\alpha_1}(x)\right)^2d\mu_{\alpha_1}(x)=\int_0^{M}\left(\varphi_{\mu_1}^{\alpha_1}(x)\right)^2x^{2\alpha_1+1}dx\leq CM\mu_1^{-1/2}.
\end{equation*}
If $\mu_1<M^2$, the estimate is trivial because $\lVert\varphi_{\mu_1}^{\alpha_1}(x)\rVert_{L^2([0,\infty),d\mu_{\alpha_1}(x))}=1$. While $\mu_1>M^2$, first by \eqref{Define} we observe that  $\left(\varphi_{\mu_1}^{\alpha_1}(x)\right)^2x^{2\alpha_1+1}=\left(\sqrt{2x}\mathscr{L}_{\mu_1}^{\alpha_1}(x^2)\right)^2$, then by equation \eqref{L1.1} in Lemma \ref{L}, we get that there exists a constant $C>0$ such that $\left(\varphi_{\mu_1}^{\alpha_1}(x)\right)^2x^{2\alpha_1+1}\leq C\mu_1^{-1/2}$ provided that $x\in[0,M]$ and $\mu_1>M^2$. Thus, we have completed the proof of Lemma \ref{Trace}.
\end{proof}
In order to make the estimate \eqref{Trace1.1} useful in proving  square function estimate, we need an extended version of \eqref{Trace1.1}. For any function $F$ with support in $[0,1]$ and $2\leq q<\infty$, we define (see \cite[(2.6)]{CD})
\begin{equation*}
\lVert F\rVert_{N^2,q}:=\left(\frac{1}{N^2}\sum_{i=1}^{N^2}\sup_{\xi\in[\frac{i-1}{N^2},\frac{i}{N^2})}\arrowvert F(x)\arrowvert ^{q}\right)^{1/q},\quad\quad N\in \mathbb{N}.
\end{equation*}
By using the same argument as in \cite{CD}, we obtain the following extended result.
\begin{lemma}\label{ETrace}
For $N\in \mathbb{N}$, let $F$ be a function supported in $[N/4,N]$ and for all $f\in L^2(\mathbb{R}_{+}^d,d\mu_{\alpha}(x))$. 
If $\beta>1$, then there exists a constant $C_\beta$ independent of $F$ and $f$ such that
\begin{equation}\label{ETrace1}
\int_{\mathbb{R}_{+}^{d}}\arrowvert F(\sqrt{L_{\alpha}})f(x)\arrowvert^2(1+\arrowvert x\arrowvert)^{-\beta}d\mu_{\alpha}(x)\leq C N\lVert \delta_NF\rVert_{N^2,2}^2\int_{\mathbb{R}_{+}^{d}}\arrowvert f(x)\arrowvert ^2d\mu_{\alpha}(x).
\end{equation}
If $0<\beta\leq1$, then for every $\varepsilon>0$, there exists a constant $C_{\beta,\varepsilon}$ independent of $F$ and $f$ such that
\begin{equation}\label{ETrace2}
\int_{\mathbb{R}_{+}^{d}}\arrowvert F(\sqrt{L_{\alpha}})f(x)\arrowvert^2(1+\arrowvert x\arrowvert)^{-\beta}d\mu_{\alpha}(x)\leq C_{\beta,\varepsilon} N^{\frac{\beta}{1+\varepsilon}}\lVert \delta_NF\rVert_{N^2,2\beta^{-1}(1+\varepsilon)}^2\int_{\mathbb{R}_{+}^{d}}\arrowvert f(x)\arrowvert ^2d\mu_{\alpha}(x).
\end{equation}
Where $\delta_{N}F(x)$ is defined by $F(Nx)$.
\end{lemma}
Regarding the proof of this lemma, we will make two points: (1) the proof of \eqref{ETrace1} is straightforward, and (2) the conclusion of \eqref{ETrace2} follows from \eqref{ETrace1} combined with a bilinear interpolation theorem. We refer the reader to \cite[p.12-15]{CD} for a detailed discussion.
\section{Proof of sufficiency part of Theorem \ref{T1.2} }\label{Section3}
In this section, we recall the standard reduction procedure (see \cite[p.17]{CD}), which reduces Theorem \ref{T1.2} to showing a square function estimate.
\subsection{Square function estimate}
We begin by recalling that
\begin{equation}\label{SI}
S_{*}^{\lambda}(L_{\alpha})f(x)\leq C_{\lambda,\rho}\sup_{0<R<\infty}\left(\frac{1}{R}\int_0^{R}\arrowvert S_t^{\rho}(L_{\alpha})f(x)\arrowvert^2dt\right)^{1/2}
\end{equation}
provided that $\rho>-1/2$ and $\lambda>\rho +1/2$. This was shown in \cite[pp.278-279]{SW} (see also \cite[p.17]{CD}). Via a dyadic decomposition, we write $x_{+}^{\rho}=\sum_{k\in\mathbb{Z}}2^{-k\rho}\phi^{\rho}(2^kx)$ for some $\phi^{\rho}\in C_{c}^{\infty}(2^{-3},2^{-1})$. Thus
\begin{equation*}
(1-\arrowvert \xi\arrowvert^2)_{+}^{\rho}=:\phi_0^{\rho}(\xi)+\sum_{k=1}^{\infty}2^{-k\rho}\phi_{k}^{\rho}(\xi),
\end{equation*}
where $\phi_{k}^{\rho}=\phi^{\rho}(2^k(1-\arrowvert \xi\arrowvert^2)), k\geq1$. We also note that supp\,$\phi_{0}^{\rho}\subseteq\{\xi:\arrowvert\xi\arrowvert\leq 7\times 2^{-3}\}$ and supp\,$\phi_{k}^{\rho}\subseteq\{\xi:1-2^{-1-k}\leq\arrowvert\xi\arrowvert\leq 1-2^{-k-3}\}$. Using \eqref{SI}, for $\lambda>\rho+1/2$ we have
\begin{equation}\label{S1}
S_{*}^{\lambda}(L_{\alpha})f(x)\leq C\left(\sup_{0<R<\infty}\frac{1}{R}\int_0^{R}\arrowvert \phi_{0}^{\rho}(t^{-1}\sqrt{L_{\alpha}})f(x)\arrowvert^2dt\right)^{1/2}+C\sum_{k=1}^{\infty}2^{-k\rho}\left(\int_0^{\infty}\arrowvert \phi_{k}^{\rho}(t^{-1}\sqrt{L_{\alpha}})f(x)\arrowvert^2\frac{dt}{t}\right)^{1/2}.
\end{equation}
The square function $\mathfrak{S}_{\delta}$ is given by
\begin{equation*}\label{S2}
\mathfrak{S}_{\delta}f(x)=\left(\int_0^{\infty}\arrowvert \phi\left(\delta^{-1}\left(1-\frac{L_{\alpha}}{t^2}\right)\right)f(x)\arrowvert^2\frac{dt}{t}\right)^{1/2}\quad,\quad 0<\delta<\frac{1}{2},
\end{equation*}
where $\phi$ is a fixed $C^{\infty}$ function supported in $[2^{-3},2^{-1}]$ with $\arrowvert \phi\arrowvert\leq 1$.
The proof of the sufficiency part of Theorem \ref{T1.2} now reduces to proving the following proposition, which is a weighted $L^2$-estimate for the square function $\mathfrak{S}_{\delta}$.
\begin{proposition}\label{SP}
Fix $\alpha\in(-1,\infty)^d$, let $0<\delta<1/2$, $0<\varepsilon\leq1/2$, and let $0\leq\beta<2(\arrowvert \alpha\arrowvert_1+d)$. Then, there exists a constant $C>0$, independent of $f$ and $\delta$, such that
\begin{equation}\label{S3}
\int_{\mathbb{R}_{+}^{d}}\arrowvert\mathfrak{S}_{\delta}f(x)\arrowvert^2(1+\arrowvert x\arrowvert)^{-\beta}d\mu_{\alpha}(x)\leq C\delta A_{\beta,d}^{\varepsilon}(\delta)\int_{\mathbb{R}_{+}^{d}}\arrowvert f(x)\arrowvert^2(1+\arrowvert x\arrowvert)^{-\beta}d\mu_{\alpha}(x),
\end{equation}
where
\begin{equation}\label{S4}
A_{\beta,d}^{\varepsilon}(\delta):=\left\{\begin{array}{lr}
\delta^{-\varepsilon},\qquad\quad 0\leq \beta< 2(\arrowvert \alpha\arrowvert_1+d),\quad if\quad 2(\arrowvert \alpha\arrowvert_1 +d)\leq 1,\\
\delta^{\frac{1}{2}-\frac{\beta}{2}},\,\,\qquad 1< \beta< 2(\arrowvert \alpha\arrowvert_1 +d),\quad if\quad 2(\arrowvert \alpha\arrowvert_1+d)\geq 1.
\end{array}
\right.
\end{equation}
\end{proposition}
Once we have the estimate \eqref{S3}, the proof of Theorem \ref{T1.2} follows from the argument outlined below. \\
\par\textbf{Proof of the sufficiency part of Theorem \ref{T1.2}.}
 Since $\lambda>0$, we choose $\eta>0$ (which will be taken arbitrarily small later) such that $\lambda-\eta>0$. We then set $\rho=\lambda-\eta-\frac{1}{2}$. With this choice of $\rho$ we can use \eqref{S1}. We first estimate the first term on the right-hand side of \eqref{S1}. Since $(1+\arrowvert x\arrowvert)^{-\beta}\in A_2(\mu_{\alpha})$ and the heat kernel of $L_{\alpha}$ satisfies the Gaussian bound, the argument from \cite[Lemma 3.1]{CLSL} shows that
\begin{equation}
\left\lVert \sup_{0<t<\infty}\arrowvert \phi_{0}^{\rho}(t^{-1}\sqrt{L_{\alpha}})f\arrowvert \right\rVert_{L^2(\mathbb{R}_{+}^d,(1+\arrowvert x\arrowvert)^{-\beta}d\mu_{\alpha})}\leq \lVert \mathscr{M} f\rVert_{L^2(\mathbb{R}_{+}^d,(1+\arrowvert x\arrowvert)^{-\beta}d\mu_{\alpha})}\leq C\lVert f\rVert_{L^2(\mathbb{R}_{+}^d,(1+\arrowvert x\arrowvert)^{-\beta}d\mu_{\alpha})}.
\end{equation} 
where $\mathscr{M}$ is the Hardy-Littlewood maximal operator defined in \eqref{HLM}. Thus, it suffices to show that the operator
\begin{equation*}
\sum_{k=1}^{\infty}2^{-k\rho}\left(\int_0^{\infty}\arrowvert \phi_{k}^{\rho}(t^{-1}\sqrt{L_{\alpha}})f(x)\arrowvert^2\frac{dt}{t}\right)^{1/2}
\end{equation*}
is bounded on $L^2(\mathbb{R}_{+}^d,(1+\arrowvert x\arrowvert)^{-\beta}d\mu_{\alpha})$. If $ 2(\arrowvert \alpha\arrowvert_1 +d)\leq 1$, then using Minkowski's inequlaity and the estimate \eqref{S3} with $\delta=2^{-k}$, we obtain, by our choice of $\rho$,
\begin{equation*}
\begin{aligned}
&\left\lVert\sum_{k=1}^{\infty}2^{-k\rho}\left(\int_0^{\infty}\arrowvert \phi_{k}^{\rho}(t^{-1}\sqrt{L_{\alpha}})f(x)\arrowvert^2\frac{dt}{t}\right)^{1/2}\right\rVert_{L^2(\mathbb{R}_{+}^d,(1+\arrowvert x\arrowvert)^{-\beta}d\mu_{\alpha})}\\
&\leq C\sum_{k=1}^{\infty}2^{-k(\lambda-\eta-\varepsilon/2)}\lVert f\rVert_{L^2(\mathbb{R}_{+}^d,(1+\arrowvert x\arrowvert)^{-\beta}d\mu_{\alpha})}.
\end{aligned}
\end{equation*}
Taking $\eta$ and $\varepsilon$ small enough, the right hand side is bounded by $C\lVert f\rVert_{L^2(\mathbb{R}_{+}^d,(1+\arrowvert x\arrowvert)^{-\beta}d\mu_{\alpha})}$. Hence, this gives the desired boundedness of $S_{*}^{\lambda}(L_{\alpha})$ on $ L^2(\mathbb{R}_{+}^d,(1+\arrowvert x\arrowvert)^{-\beta}d\mu_{\alpha})$ for $ 2(\arrowvert \alpha\arrowvert_1 +d)\leq 1$ and $0\leq\beta< 2(\arrowvert \alpha\arrowvert_1 +d)$. The same argument extends to the case $ 2(\arrowvert \alpha\arrowvert_1 +d)> 1$ and $\beta>1$, and the range of $\beta$ can be extended to $\beta>0$ by interpolation. For more details, one can refer to \cite[p.18]{CD}.$\hfill{\Box}$\\
\par Now, to complete the proof of the sufficiency part of Theorem \ref{T1.2}, it remains to prove Proposition \ref{SP}.
\subsection{Proof of the square function estimate}
In this subsection, we establish Proposition \ref{SP}. To this end, we decompose $\mathfrak{S}_{\delta}$ into high- and low-frequency parts. First, for every fixed $\alpha\in(-1,\infty)^{d}$, if $\arrowvert \alpha\arrowvert_1\geq 0$, in this case we set
\begin{equation*}
\mathfrak{S}_{\delta}^{l}f(x)=\left(\int_{1/2}^{\delta^{-1/2}}\arrowvert \phi\left(\delta^{-1}\left(1-\frac{L_{\alpha}}{t^2}\right)\right)f(x)\arrowvert^2\frac{dt}{t}\right)^{1/2},
\end{equation*}
\begin{equation*}
\mathfrak{S}_{\delta}^{h}f(x)=\left(\int_{\delta^{-1/2}}^{\infty}\arrowvert \phi\left(\delta^{-1}\left(1-\frac{L_{\alpha}}{t^2}\right)\right)f(x)\arrowvert^2\frac{dt}{t}\right)^{1/2}.
\end{equation*}
Second, if $-d<\arrowvert \alpha\arrowvert_1< 0$, in this case we set
\begin{equation*}
\mathfrak{S}_{\delta}^{l}f(x)=\left(\int_{2^{[\log_2\sqrt{\arrowvert \alpha\arrowvert_1+d}]-1}}^{\delta^{-1/2}}\arrowvert \phi\left(\delta^{-1}\left(1-\frac{L_{\alpha}}{t^2}\right)\right)f(x)\arrowvert^2\frac{dt}{t}\right)^{1/2},
\end{equation*}
\begin{equation*}
\mathfrak{S}_{\delta}^{h}f(x)=\left(\int_{\delta^{-1/2}}^{\infty}\arrowvert \phi\left(\delta^{-1}\left(1-\frac{L_{\alpha}}{t^2}\right)\right)f(x)\arrowvert^2\frac{dt}{t}\right)^{1/2}.
\end{equation*}
Here, the notation $[\cdot]$ denotes the floor function.
If $\arrowvert \alpha\arrowvert_1\geq 0$, then the first eigenvalue of the Laguerre operator $L_\alpha$ is larger than or equal to 1, $\phi\left(\delta^{-1}\left(1-\frac{L_{\alpha}}{t^2}\right)\right)$=0 if $t\leq1$ because supp\,$\phi\subset(2^{-3},2^{-1})$ and $\delta\leq1/2$. Similarly, if $-d<\arrowvert \alpha\arrowvert_1< 0$, the first eigenvalue of the Laguerre operator $L_\alpha$ is larger than or equal to $2(\arrowvert\alpha\arrowvert_1+d)>0$, and $\phi\left(\delta^{-1}\left(1-\frac{L_{\alpha}}{t^2}\right)\right)$=0 if $t\leq \sqrt{2(\arrowvert\alpha\arrowvert_1+d)}$. Thus in both cases, we have
$$\mathfrak{S}_{\delta}f(x)\leq\mathfrak{S}_{\delta}^{l}f(x)+\mathfrak{S}_{\delta}^{h}f(x).$$
\par Now, in order to prove Proposition \ref{SP}, it is sufficient to show the following lemma.
\begin{lemma}\label{SL}
Let $A_{\beta,d}^{\varepsilon}(\delta)$ be given by \eqref{S4}. Then, for all $0<\delta<1/2$ and $0<\varepsilon\leq1/2$, we have the following estimates:
\begin{equation}\label{low}
\int_{\mathbb{R}_{+}^{d}}\arrowvert\mathfrak{S}_{\delta}^{l}f(x)\arrowvert^2(1+\arrowvert x\arrowvert)^{-\beta}d\mu_{\alpha}(x)\leq C\delta A_{\beta,d}^{\varepsilon}(\delta)\int_{\mathbb{R}_{+}^{d}}\arrowvert f(x)\arrowvert^2(1+\arrowvert x\arrowvert)^{-\beta}d\mu_{\alpha}(x),
\end{equation}
\begin{equation}\label{high}
\int_{\mathbb{R}_{+}^{d}}\arrowvert\mathfrak{S}_{\delta}^{h}f(x)\arrowvert^2(1+\arrowvert x\arrowvert)^{-\beta}d\mu_{\alpha}(x)\leq C\delta A_{\beta,d}^{\varepsilon}(\delta)\int_{\mathbb{R}_{+}^{d}}\arrowvert f(x)\arrowvert^2(1+\arrowvert x\arrowvert)^{-\beta}d\mu_{\alpha}(x).
\end{equation}
\end{lemma}
Regarding the proof of this lemma, let us first make a few remarks. Both of the proofs of the estimates \eqref{low} and \eqref{high} rely on the generalized trace Lemma \ref{ETrace}. The primary distinction lies in the fact that, for the estimate in \eqref{low}, we also employ the estimate in \eqref{weighted}, which proves particularly effective for the low-frequency part. This approach was originally introduced in \cite{CD}. For the estimate in \eqref{high}, we utilize a spatial localization argument based on the finite speed of propagation of the Laguerre wave operator $\cos(t\sqrt{L_{\alpha}})$. The strategy of the proof was first developed in \cite{Carbery1} and has since been applied to related problems, such as those in \cite{CD, CLSL}. The choice of $\delta^{-\frac{1}{2}}$ in the definition of $\mathfrak{S}_{\delta}^{l}$, $\mathfrak{S}_{\delta}^{h}$ is made by optimizing the estimates which result from two different approaches, see \cite[Remark 3.2]{CD}. Given that the proof of this lemma is quite intricate and plays a central role in this paper, we defer the detailed proof to the appendix for the sake of clarity of this paper.

\section{Sharpness of the summability indices}\label{Section4}
The sharpness of the summability indices (up to the endpoints) in Theorem \ref{T1.1} and Theorem \ref{T1.2} follows from the two propositions below.
\begin{proposition}\label{Sp1}
Let $2\arrowvert\alpha\arrowvert_1+2d>1$ and $p>\frac{4\arrowvert \alpha\arrowvert_1+4d}{2\arrowvert\alpha\arrowvert_1+2d-1}$. If
\begin{equation}\label{Sp1.1}
\sup_{R>0}\arrowvert S_{R}^{\lambda}(L_{\alpha})f\arrowvert<\infty\quad a.e
\end{equation}
for all $f\in L^{p}(\mathbb{R}_{+}^d,d\mu_{\alpha}(x))$, then we have $\lambda\geq\lambda(\alpha,p)/2$.
\end{proposition}
This proposition demonstrates that the summability index in Theorem \ref{T1.1} is sharp up to the endpoint, as we are assuming $\lambda\geq0$.
\begin{proposition}\label{Sp2}
Let $0\leq\beta<2\arrowvert\alpha\arrowvert_1+2d$. Suppose that
\begin{equation}\label{Sp2.1}
\sup_{R>0}\lVert S_{R}^{\lambda}(L_{\alpha})f\rVert_{L^2(\mathbb{R}_{+}^d,(1+\arrowvert x\arrowvert)^{-\beta}d\mu_{\alpha})\rightarrow L^2(\mathbb{R}_{+}^d,(1+\arrowvert x\arrowvert)^{-\beta}d\mu_{\alpha})}<\infty,
\end{equation}
then, we have $\lambda\geq \max\{\frac{\beta-1}{4},0\}$.
\end{proposition}
Since 
\begin{equation*}
\sup_{R>0}\lVert S_{R}^{\lambda}(L_{\alpha})f\rVert_{L^2(\mathbb{R}_{+}^d,(1+\arrowvert x\arrowvert)^{-\beta}d\mu_{\alpha})}\leq \lVert \sup_{R>0}\arrowvert S_{R}^{\lambda}(L_{\alpha})f\arrowvert\rVert_{L^2(\mathbb{R}_{+}^d,(1+\arrowvert x\arrowvert)^{-\beta}d\mu_{\alpha})},
\end{equation*}
it is clear that Proposition \ref{Sp2} implies the necessity part of Theorem \ref{T1.2}. We now proceed to prove these two propositions one by one.
\subsection{Proof of Proposition \ref{Sp1}}
To prove Proposition \ref{Sp1}, we use the following consequence of the Nikishin-Maurey theorem. 
\begin{theorem}\label{TNM} \cite[Theorem 4.3]{CD} Let $1\leq p<\infty$ and $(X,\mu)$ be a $\sigma$-finite measure space. Suppose that $\{T_m\}_{m\in \mathbb{N}}$ is a sequence of linear operators which are continuous from $L^{p}(X,\mu)$ to $L^{0}(X,\mu)$ of all measurable functions on $(X,\mu)$ (with the topology of convergence in measure). Assume that $T^{*}f=\sup_{m}\arrowvert T_m f\arrowvert<\infty$ a.e. whenever $f\in L^{p}(X,\mu)$, then there exists a measurable function $\omega>0$ a.e. such that
\begin{equation*}
\int_{\{x: \arrowvert T^{*}f>\delta \}}\omega(x)d\mu(x)\leq C\left(\frac{\lVert f\rVert_{p}}{\alpha}\right)^{q}\quad,\qquad\alpha>0
\end{equation*}
for all $f\in L^{p}(X,\mu)$ with $q$=min$(p,2)$.
\end{theorem}
In the remaining part of this subsection we mainly work with radial functions. For a given function $f$ on $[0,\infty)$, we set $f_{rad}(x):=f(\arrowvert x\arrowvert)$. Clearly, $f_{rad}\in L^{p}(\mathbb{R}_{+}^d,d\mu_{\alpha})$ if and only if $f\in L^p([0,\infty), r^{\arrowvert 2\alpha\arrowvert_1+2d-1}dr)$. Now, to prove Proposition \ref{Sp1}, it suffices to show the following Proposition \ref{Sp3}.
\begin{proposition}\label{Sp3}
Let $2\leq p<\infty$ and $\omega$ be a measurable function on $(0,\infty)$ with $\omega>0$ almost everywhere. Suppose that
\begin{equation}\label{Sp3.1}
\sup_{R>0}\sup_{\delta>0}\delta\left(\int_{\{x\in\mathbb{R}_{+}^{d}:\arrowvert S_{R}^{\lambda}(L_{\alpha})f\arrowvert>\delta\}}\omega(\arrowvert x\arrowvert)d\mu_{\alpha}(x)\right)^{1/2}\leq C\lVert f\rVert_{L^p(\mathbb{R}_{+}^d,d\mu_{\alpha}(x))}
\end{equation}
for all radial functions $f\in L^{p}(\mathbb{R}_{+}^d,d\mu_{\alpha}(x))$, then we have $\lambda\geq\lambda(\alpha,p)/2$.
\end{proposition}
\par \textbf{Proof of Proposition \ref{Sp1}.} The proof follows the same approach as in \cite{CD}. For the sake of completeness, we provide a sketch of the proof here.  Let $R_m$ be any enumeration of the rational numbers in $(0,\infty)$. Then we have $S_{*}^{\lambda}(L_{\alpha})f(x)=\sup_{m}\arrowvert S_{R_{m}}^{\lambda}(L_{\alpha})f(x)$. We now restrict the operator 
$S_{R_{m}}^{\lambda}(L_{\alpha})$ to the set of radial functions and we may view $S_{R_{m}}^{\lambda}(L_{\alpha})$ as an operator from  $L^p([0,\infty), r^{\arrowvert 2\alpha\arrowvert_1+2d-1}dr)$ to itself. More precisely, Let us denote by $T_m$ the mapping
\begin{equation}
f\rightarrow S_{R_{m}}^{\lambda}(L_{\alpha}) f_{rad}
\end{equation}
for $f\in L^p([0,\infty), r^{\arrowvert 2\alpha\arrowvert_1+2d-1}dr)$. Since $\sup_{R>0}\arrowvert S_{R}^{\lambda}(L_{\alpha})f\arrowvert<\infty$ a.e. for all $f\in L^{p}(\mathbb{R}_{+}^d,d\mu_{\alpha}(x))$ by assumption \eqref{Sp1.1}, we clearly have
\begin{equation*} 
\sup_{m}\arrowvert T_mf(r)\arrowvert<\infty,\quad a.e.\quad r\in [0,\infty)
\end{equation*}
for all $f\in L^p([0,\infty), r^{\arrowvert 2\alpha\arrowvert_1+2d-1}dr)$. Now we take $(X,\mu)=([0,\infty),r^{\arrowvert 2\alpha\arrowvert_1+2d-1}dr)$. Clearly, each operator $T_m$ is continuous $L^p(X,\mu)$ to itself and so is $T_m$ from  $L^p(X,\mu)$ to  $L^0(X,\mu)$. Then it follows from Theorem \ref{TNM} that there exists a weight function $\omega>0$ such that $f\rightarrow\sup_m\arrowvert T_m f\arrowvert$ is bounded from $L^p([0,\infty),r^{\arrowvert 2\alpha\arrowvert_1+2d-1}dr)$ to $L^{2,\infty}([0,\infty),\omega (r)r^{\arrowvert 2\alpha\arrowvert_1+2d-1}dr)$. Therefore,  we get a weight $\omega$ which satisfies \eqref{Sp3.1} for all radial functions  $f\in L^{p}(\mathbb{R}_{+}^d,d\mu_{\alpha}(x))$, by which we conclude that 
$\lambda\geq\lambda(\alpha,p)/2$. $\hfill{\Box}$
\\
\par Now, it remains to prove Proposition \ref{Sp3}. For this purpose, we need several estimates for Laguerre functions . Define $$\tilde{L}_n^{\alpha}(x)=\left(\frac{\Gamma(\alpha+1)\Gamma(n+1)}{\Gamma(n+\alpha+1)}\right)^{1/2}L_n^{\alpha}(x)$$ for $n\in \mathbb{N}$ and $\alpha>-1$. For a multiindex $\mu=(\mu_1,...,\mu_d)\in \mathbb{N}^{d}$ and $\alpha=(\alpha_1,...,\alpha_d)\in (-1,\infty)^d$, the multiple function $\tilde{L}_{\mu}^{\alpha}(x)$ are defined by the formula
\begin{equation*}
\tilde{L}_{\mu}^{\alpha}(x)=\prod_{j=1}^{d}\tilde{L}_{\mu_j}^{\alpha_j}(x).
\end{equation*}
We define function $\tilde{P}_n(x,y)$ by
\begin{equation*}
\tilde{P}_n(x,y)=\sum_{\arrowvert \mu\arrowvert_1=n}\tilde{L}_{\mu}^{\alpha}(x)\tilde{L}_{\mu}^{\alpha}(y).
\end{equation*}
Let $\Sigma^{d-1}$ denote the simplex $\Sigma^{d-1}=\{x\in \mathbb{R}_{+}^d: x_1+\cdot\cdot\cdot+x_d=1\}$. Then from the proof of Theorem 2.4 in \cite{XY}, we have the identity
\begin{equation}\label{Iden}
\int_{\Sigma^{d-1}}\tilde{P}_n(x,ry)y^{\alpha}dy=\frac{\prod_{i=1}^{d}\Gamma(\alpha_i+1)}{\Gamma(\arrowvert\alpha\arrowvert_1+d)}\tilde{L}_n^{\arrowvert\alpha\arrowvert_1+d-1}(r)\tilde{L}_n^{\arrowvert\alpha\arrowvert_1+d-1}(x_1+\cdot\cdot\cdot+x_d).
\end{equation}
Recall that $P_n$ is the Laguerre spectral projection operator defined by \eqref{Projection}. We denote the kernel of $P_n$ by $P_n(x,y)$, and we have
$$P_n(x,y)=\sum_{\arrowvert \mu\arrowvert_1=n}\varphi_{\mu}^{\alpha}(x)\varphi_{\mu}^{\alpha}(y).$$
Now, we proceed to establish several preparatory lemmas.
\begin{lemma}\label{L4.1}
 If $f(x)=f_0(\arrowvert x\arrowvert)$ on $\mathbb{R}_{+}^d$, then we have
 \begin{equation*}
 P_n(f)(\arrowvert x\arrowvert)=R_{n}^{\arrowvert \alpha\arrowvert_1+d-1}(f_0)\varphi_{n}^{\arrowvert\alpha\arrowvert_1+d-1}(\arrowvert x\arrowvert),
 \end{equation*}
 where
 \begin{equation*}
 R_{n}^{\arrowvert \alpha\arrowvert_1+d-1}(f_0)=\int_{0}^{\infty} f_0(r)\varphi_{n}^{\arrowvert\alpha\arrowvert_1+d-1}(r)r^{2(\arrowvert\alpha\arrowvert_1+d)-1}dr.
 \end{equation*}
\end{lemma}
\begin{proof} 
Let $S_{+}^{d-1}$ denote the surface $S_{+}^{d-1}=\{x\in \mathbb{R}_{+}^d: \arrowvert x\arrowvert ^2=x_1^{2}+\cdot\cdot\cdot+x_d^{2}=1\}$. For every $f(x)=f_0(\arrowvert x\arrowvert)$,  we have
\begin{equation}\label{L4.1.1}
\begin{aligned}
 P_n(f)(x)&=\int_{\mathbb{R}_{+}^d}f_0(\arrowvert y\arrowvert)P_n(x,y)y^{2\alpha+1}dy\\
 &=\int_{0}^{\infty}f_0(r)r^{2\arrowvert \alpha\arrowvert_1 +2d-1}\int_{S_{+}^{d-1}}P_n(x,r\omega)\omega^{2\alpha+1}d\omega dr.
 \end{aligned}
\end{equation}
It is straightforward to check that
\begin{equation}\label{L4.1.2}
P_n(x,y)=\frac{2^d}{\prod_{i=1}^{d}\Gamma(\alpha_i+1)}e^{-\frac{\arrowvert x\arrowvert^2}{2}}\tilde{P}_n((x_1^2,\cdot\cdot\cdot,x_d^2),(y_1^2,\cdot\cdot\cdot,y_d^2))e^{-\frac{\arrowvert y\arrowvert^2}{2}}.
\end{equation}
Putting \eqref{L4.1.2} into \eqref{L4.1.1} and let $z=r^2(\omega_1^2,\cdot\cdot\cdot,\omega_d^2)$ we obtain
\begin{equation*}
P_n(f)(x)=\frac{2}{\prod_{i=1}^{d}\Gamma(\alpha_i+1)}\int_{0}^{\infty}f_0(r)e^{-\frac{r^2}{2}}r^{2\arrowvert \alpha\arrowvert_1 +2d-1}\int_{\Sigma^{d-1}}e^{-\frac{\arrowvert x\arrowvert^2}{2}}P_n((x_1^2,\cdot\cdot\cdot,x_d^2),r^2z)z^{\alpha}dzdr,
\end{equation*}
thus by identity \eqref{Iden}, we get
\begin{equation*}
\begin{aligned}
P_n(f)(x)&=\frac{2}{\Gamma(\arrowvert\alpha\arrowvert_1+d)}\int_{0}^{\infty}f_0(r)e^{-\frac{r^2}{2}}\tilde{L}_n^{\arrowvert\alpha\arrowvert_1+d-1}(r^2)r^{2\arrowvert \alpha\arrowvert_1 +2d-1}dr \tilde{L}_n^{\arrowvert\alpha\arrowvert_1+d-1}(\arrowvert x\arrowvert ^2)e^{-\frac{\arrowvert x\arrowvert^2}{2}}\\
&=\int_{0}^{\infty} f_0(r)\varphi_{n}^{\arrowvert\alpha\arrowvert_1+d-1}(r)r^{2(\arrowvert\alpha\arrowvert_1+d)-1}dr\varphi_{n}^{\arrowvert\alpha\arrowvert_1+d-1}(\arrowvert x\arrowvert).
\end{aligned}
\end{equation*}
Thus we have completed the proof of this lemma.
\end{proof}
\begin{lemma}\label{L4.3}
For fixed $\alpha\in(-1,\infty)^d$, let $\omega$ be a measurable function on $(0,\infty)$ with $\omega>0$ almost everywhere. If $n\in \mathbb{N}$ is large enough, we have
\begin{equation*}\label{L4.3.1}
\sup_{\delta}\delta\left(\int_{\{x\in\mathbb{R}_{+}^{d}:\arrowvert \phi_{n}^{\arrowvert\alpha\arrowvert_1+d-1}(\arrowvert x\arrowvert)\arrowvert>\delta\}}\omega(\arrowvert x\arrowvert)d\mu_{\alpha}(x)\right)^{1/2}\geq C_{\alpha} n^{-1/4}
\end{equation*}
with a constant $C_{\alpha}>0$ independent of $n$.
\end{lemma}
\begin{proof}
Recalling the following asymptotic property of Laguerre functions (see \cite[7.4, p.453]{M}):
\begin{equation}\label{PL4.3.1}
\mathscr{L}_{n}^{a}=\frac{(2/\pi)^{1/2}}{(\nu r)^{1/4}}\left(\cos\left(\sqrt{\nu r}-\frac{a\pi}{2}-\frac{\pi}{4}\right)+O(\nu^{-1/2}r^{-1/2})\right),
\end{equation}
where $\nu=4n+2a+2$ and $1/\nu\leq r\leq1$. By \cite[4.6]{CD} we have that there exists a constant $C_{*} $, independent of $\nu$, such that
\begin{equation}\label{PL4.3.2}
\arrowvert E(\nu)\arrowvert\geq C_{*},
\end{equation}
where 
\begin{equation*}
E(\nu):=\left\{r\in[\frac{1}{2},1]:\left\arrowvert \cos\left(\sqrt{\nu r}-\frac{\alpha\pi}{2}-\frac{\pi}{4}\right)\right\arrowvert\geq \frac{\sqrt{2}}{2}\right\}.
\end{equation*}
Since $\arrowvert \phi_n^{\arrowvert \alpha\arrowvert_1 +d-1}(r)\arrowvert=\sqrt{2}\arrowvert\mathscr{L}_{n}^{\arrowvert \alpha\arrowvert_1 +d-1}(r^2)r^{-(\arrowvert \alpha\arrowvert_1 +d-1)}\arrowvert$, then by \eqref{PL4.3.1}, we have
\begin{equation*}\arrowvert \phi_n^{\arrowvert \alpha\arrowvert_1 +d-1}(r)\arrowvert=\sqrt{2}\arrowvert\mathscr{L}_{n}^{\arrowvert \alpha\arrowvert_1 +d-1}(r^2)r^{-(\arrowvert \alpha\arrowvert_1 +d-1)}\arrowvert\geq Cn^{-1/4}
\end{equation*}
for all $r\in E(\nu)$. On the other hand, since $\omega(r)>0$ a.e. $r\in [1/2,1]$, there exists a subset $F$ of $[1/2,1]$ and a constant $c_0>0$ such that $\arrowvert F\arrowvert>1/2-C_*/2$ and  $\omega(r)\geq c_0$ for all $r\in F$ (see also \cite[p.33]{CD}). Then, we have $\arrowvert E(\nu)\cap F\arrowvert\geq \arrowvert E(\nu)\arrowvert \arrowvert F\arrowvert-\arrowvert[1/2,1]\arrowvert\geq C_*/2$ and
\begin{equation*}
\begin{aligned}
\sup_{\delta}\delta^2\int_{\{x\in\mathbb{R}_{+}^{d}:\arrowvert \phi_{n}^{\arrowvert\alpha\arrowvert_1+d-1}(\arrowvert x\arrowvert)\arrowvert>\delta\}}\omega(\arrowvert x\arrowvert)d\mu_{\alpha}(x)&\geq \sup_{\delta}\delta^2\int_{\{E(\nu)\cap F:\arrowvert \phi_{n}^{\arrowvert\alpha\arrowvert_1+d-1}( r)\arrowvert>\delta\}}\omega(r)r^{2\arrowvert\alpha\arrowvert_1+2d-1}dr\\
&\geq c_0\min\{(1/2)^{2\arrowvert\alpha\arrowvert_1+2d-1},1\}\sup_{\delta}\delta^2\int_{\{E(\nu)\cap F:\arrowvert \phi_{n}^{\arrowvert\alpha\arrowvert_1+d-1}( r)\arrowvert>\delta\}}dr.
\end{aligned}
\end{equation*}
Particularly, taking $\delta= Cn^{-1/4}/2$ and using that $\arrowvert \phi_n^{\arrowvert \alpha\arrowvert_1 +d-1}(r)\arrowvert\geq Cn^{-1/4}$ for $r\in E(\nu)$ and $\arrowvert E(\nu)\cap F\arrowvert\geq C_{*/2}$, we have 
\begin{equation*}
\begin{aligned}
\sup_{\delta}\delta^2\int_{\{x\in\mathbb{R}_{+}^{d}:\arrowvert \phi_{n}^{\arrowvert\alpha\arrowvert_1+d-1}(\arrowvert x\arrowvert)\arrowvert>\delta\}}\omega(\arrowvert x\arrowvert)d\mu_{\alpha}(x)
& \geq c_0\min\{(1/2)^{2\arrowvert\alpha\arrowvert_1+2d-1},1\}(Cn^{-1/4}/2)^2\arrowvert E(\nu)\cap F\arrowvert\\
&\geq c_0\min\{(1/2)^{2\arrowvert\alpha\arrowvert_1+2d-1},1\}2^{-3}C^2C_{*} n^{-1/2},
\end{aligned}
\end{equation*}
which implies that \eqref{L4.3.1} holds for $C_{\alpha}=\sqrt{c_0\min\{(1/2)^{2\arrowvert\alpha\arrowvert_1+2d-1},1\}2^{-3}C^2C_{*}}$, which is independent of $n$. Thus we have completed the proof of this lemma.
\end{proof}
 Next, we recall the following estimate from \cite{TS3}.
\begin{lemma}\label{L4.2} \cite[(i) in Lemma 1.5.4]{TS3}
Let $\beta+\gamma>-1$, $\beta>-2/q$, and $1\leq q\leq 2$. Then, if $n$ is large enough, for $\gamma<2/q-1/2$ we have
\begin{equation*}
\lVert \mathscr{L}_{n}^{\beta+\gamma}(x)x^{-\gamma/2}\rVert_{L^q(\mathbb{R}_{+})}\sim n^{\frac{1}{q}-\frac{1}{2}-\frac{\gamma}{2}},
\end{equation*}
where $\mathscr{L}_{n}^{\beta+\gamma}(x)$ are normalized Laguerre function given by \eqref{Define}.
\end{lemma}
\begin{lemma}\label{L4.4}
Let $1\leq q\leq 2$. Then we have the estimate
\begin{equation}
\lVert \phi_n^{\arrowvert \alpha\arrowvert_1 +d-1}\rVert_{L^q([0,\infty), r^{2\arrowvert\alpha\arrowvert_1+2d-1})dr}\sim n^{(\arrowvert\alpha\arrowvert_1+d)(1/q-1/2)}.
\end{equation}
\end{lemma}
\begin{proof}
By \eqref{Define}, we note that $\arrowvert \phi_n^{\arrowvert \alpha\arrowvert_1 +d-1}(r)\arrowvert=\sqrt{2}\arrowvert\mathscr{L}_{n}^{\arrowvert \alpha\arrowvert_1 +d-1}(r^2)r^{-(\arrowvert \alpha\arrowvert_1 +d-1)}\arrowvert$. We take $\beta=2(\arrowvert \alpha\arrowvert_1 +d-1)/q$ and $\gamma=2(\frac{1}{2}-\frac{1}{q})(\arrowvert \alpha\arrowvert_1 +d-1)$ in Lemma \ref{L4.2} to obtain
\begin{equation*}
\int_{0}^{\infty}\arrowvert \phi_n^{\arrowvert \alpha\arrowvert_1 +d-1}(r)\arrowvert^q r^{2\arrowvert\alpha\arrowvert_1+2d-1}dr=\sqrt{2}\int_{0}^{\infty}\arrowvert\mathscr{L}_{n}^{\arrowvert \alpha\arrowvert_1 +d-1}(r^2)r^{-(\arrowvert \alpha\arrowvert_1 +d-1)}\arrowvert ^qr^{2\arrowvert\alpha\arrowvert_1+2d-1}dr\sim n^{(\arrowvert\alpha\arrowvert_1+d)(1/q-1/2)}.
\end{equation*}
\end{proof}

Now, we are ready to prove Proposition \ref{Sp3} using the argument presented in \cite{CD}.\\
\par \textbf{Proof of Proposition \ref{Sp3}.} We denote $\tilde{\omega}(x)=\omega(\arrowvert x\arrowvert)$. For a function $F$ compactly supported in $[0,\infty)$, by \cite[4.10]{CD}, we have 
\begin{equation*}\label{P}
F(L_{\alpha})f(x)=\frac{1}{\Gamma(\lambda+1)}\int_0^{\infty}F^{\lambda+1}(R)R^{\lambda}S_{\sqrt{R}}^{\lambda}(L_{\alpha})f(x)dR.
\end{equation*}
By Minkowski's inequality and the assumption \eqref{Sp3.1}, we obtain
\begin{equation*}
\begin{aligned}
\lVert F(L_{\alpha})f\rVert_{L^{2,\infty}(\tilde{\omega}(x)d\mu_{\alpha}(x),\mathbb{R}_{+}^d)}&\leq C \sup_{R>0}\lVert S_{R}^{\lambda}(L_{\alpha})f\rVert_{L^{2,\infty}(\tilde{\omega}(x)d\mu_{\alpha}(x),\mathbb{R}_{+}^d)}\int_0^{\infty}\arrowvert F^{\lambda+1}(s)s^{\lambda}ds\\
&\leq C\lVert f\rVert_{L^{p}(\mathbb{R}_{+}^d,d\mu_{\alpha})}\int_0^{\infty}\arrowvert F^{\lambda+1}(s)s^{\lambda}ds.
\end{aligned}
\end{equation*}
Let $\eta$ be a non-negative smooth function such that $\eta(0)=1$ and supp\,$\eta\subset[-1,1]$. Taking $F=\eta(\cdot-(4n+2\arrowvert \alpha\arrowvert_1+2d))$ in the above estimate, we get
\begin{equation}\label{key}
\lVert P_nf\rVert_{L^{2,\infty}(\tilde{\omega}(x)d\mu_{\alpha}(x),\mathbb{R}_{+}^d)}\leq C n^{\lambda}\lVert f\rVert_{L^{p}(\mathbb{R}_{+}^d,d\mu_{\alpha})},
\end{equation}
because $\int_0^{\infty}\arrowvert \eta^{\lambda+1}(s-n)s^{\lambda}ds\sim n^{\lambda}$, and $\eta(L_{\alpha}-(4n+2\arrowvert \alpha\arrowvert_1+2d))f=P_nf$. 
\par Now, we take
\begin{equation*}
f_n(r):=sign (\phi_n^{\arrowvert \alpha\arrowvert_1 +d-1}(r))\arrowvert\phi_n^{\arrowvert \alpha\arrowvert_1 +d-1}(r)\arrowvert^{1/(p-1)}.
\end{equation*}
Then, we deduce from Lemma \ref{L4.1} and Lemma \ref{L4.3} that
\begin{equation*}
\begin{aligned}
\lVert P_nf_n(\arrowvert \cdot\arrowvert)\rVert_{L^{2,\infty}(\mathbb{R}_{+}^{d},\tilde{\omega}d\mu_{\alpha}(x))}&=R_{n}^{\arrowvert \alpha\arrowvert_1+d-1}(f_n)\lVert\phi_n^{\arrowvert \alpha\arrowvert_1 +d-1}(\arrowvert\cdot\arrowvert)\rVert_{L^{2,\infty}(\mathbb{R}_{+}^{d},\tilde{\omega}d\mu_{\alpha}(x))}\\
&\geq Cn^{-1/4}\int_{0}^{\infty} f_n(r)\varphi_{n}^{\arrowvert\alpha\arrowvert_1+d-1}(r)r^{2(\arrowvert\alpha\arrowvert_1+d)-1}dr.
\end{aligned}
\end{equation*}
By our choice of $f_n$, it is clear that 
\begin{equation*}
\begin{aligned}
&\int_{0}^{\infty} f_n(r)\varphi_{n}^{\arrowvert\alpha\arrowvert_1+d-1}(r)r^{2(\arrowvert\alpha\arrowvert_1+d)-1}dr\\
&=\lVert \varphi_{n}^{\arrowvert\alpha\arrowvert_1+d-1}\rVert_{L^{p'}([0,\infty),r^{2(\arrowvert\alpha\arrowvert_1+d)-1}dr)}^{p'-1}\lVert \varphi_{n}^{\arrowvert\alpha\arrowvert_1+d-1}\rVert_{L^{p'}([0,\infty),r^{2(\arrowvert\alpha\arrowvert_1+d)-1}dr)}\\
&=\lVert f_n\rVert_{L^{p}([0,\infty),r^{2(\arrowvert\alpha\arrowvert_1+d)-1}dr)}\lVert \varphi_{n}^{\arrowvert\alpha\arrowvert_1+d-1}\rVert_{L^{p'}([0,\infty),r^{2(\arrowvert\alpha\arrowvert_1+d)-1}dr)}.
\end{aligned}
\end{equation*}
Thus using Lemma \ref{L4.4}, we obtain
\begin{equation*}
\lVert P_nf_n(\arrowvert \cdot\arrowvert)\rVert_{L^{2,\infty}(\mathbb{R}_{+}^{d},\tilde{\omega}d\mu_{\alpha}(x))}\geq Cn^{(\arrowvert\alpha\arrowvert_1+d)(1/q-1/2)-1/4}\lVert f_n(\arrowvert \cdot\arrowvert)\rVert_{L^{p}(\mathbb{R}_{+}^{d},d\mu_{\alpha}(x))}.
\end{equation*}
Then we combine this with \eqref{key} where we take $f=f_n(\arrowvert \cdot\arrowvert)$ to obtain
\begin{equation}\label{Fin}
n^{(\arrowvert\alpha\arrowvert_1+d)(1/q-1/2)-1/4}\lVert f_n(\arrowvert \cdot\arrowvert)\rVert_{L^{p}(\mathbb{R}_{+}^{d},d\mu_{\alpha}(x))}\leq n^{\lambda}\lVert f_n(\arrowvert \cdot\arrowvert)\rVert_{L^{p}(\mathbb{R}_{+}^{d},d\mu_{\alpha}(x))}
\end{equation}
with $C$ independent of $n$. Thus \eqref{Fin} implies $\lambda\geq (\arrowvert\alpha\arrowvert_1+d)(1/q-1/2)-1/4$ as desired by letting $k$ tend to infinity. The proof is completed.$\hfill{\Box}$
\subsection{Proof of Proposition \ref{Sp2}}
To prove Proposition \ref{Sp2}, we need to established the following weighted estimates of the Laguerre functions $\varphi_n^{\alpha}(x)$.
\begin{lemma}\label{L4.8}
For every $\alpha>-1$, let $\beta\geq 0$. Then, if $n\in \mathbb{N}$ is large enough, we have
\begin{equation}\label{L4.8.1}
\int_{0}^{\infty}\left(\varphi_n^{\alpha}(x)\right)^2(1+\arrowvert x\arrowvert)^{\beta}x^{2\alpha+1}dx\geq C n^{\beta/2},
\end{equation}
\begin{equation}\label{L4.8.2}
\int_{0}^{\infty}\left(\varphi_n^{\alpha}(x)\right)^2(1+\arrowvert x\arrowvert)^{-\beta}x^{2\alpha+1}dx\geq C \max \{n^{-\beta/2},n^{-1/2}\}.
\end{equation}
\end{lemma}
\begin{proof}
We adopt the argument of \cite[Lemma 4.9]{CD}. Note that $x^{\alpha+1/2}\varphi_n^{\alpha}(x)=\sqrt{2x}\mathscr{L}_{n}^{a}(x^2)$, by the asymptotic \eqref{L1.2} of Laguerre functions $\mathscr{L}_{n}^{a}(x)$, we have
\begin{equation}\label{PL4.8.1}
x^{\alpha+1/2}\varphi_n^{\alpha}(x)=\frac{2(-1)^{n}}{\sqrt{\pi}}(\nu-x^2)^{-\frac{1}{4}}\cos\left(\frac{\nu(2\theta-\sin 2\theta)-\pi}{4}\right)+O\left(\frac{\nu^{\frac{1}{4}}}{(\nu-x^2)^{\frac{7}{4}}}+(x^2\nu)^{-\frac{3}{4}}\right),
\end{equation}
where $\theta=\arccos(x\nu^{-1/2})$ and $1\leq x\leq\nu^{1/2}-\nu^{1/6}$.
\cite[4.26]{CD} shows that there exists a constant $C>0$ such that, for any large $\nu$,
\begin{equation*}\label{PL4.8.2}
\arrowvert E(\nu)\arrowvert\geq C\sqrt{\nu},
\end{equation*}
where
\begin{equation*}
E(\nu):=\left\{ x\in [\sqrt{\nu}/2,\sqrt{\nu}/\sqrt{2}]:\cos\left(\frac{\nu(2\theta-\sin\theta)-\pi}{4}\right)\geq\frac{\sqrt{2}}{2}\right\}
\end{equation*}
and $\theta=\arccos(x \nu^{-1/2})$. Now, the desired estimate \eqref{L4.8.1} follows because $x^{\alpha+1/2}\varphi_n^{\alpha}(x)\geq C\nu^{-1/4}$,  $x\in E(\nu) $ by \eqref{PL4.8.1}. Clearly, we also have $\int_{0}^{\infty}\left(\varphi_n^{\alpha}(x)\right)^2(1+\arrowvert x\arrowvert)^{-\beta}x^{2\alpha+1}dx\geq C  n^{-\beta/2}$. On the other hand, \eqref{PL4.3.2} tells us 
\begin{equation*}
\left\arrowvert\left\{ x\in [\frac{1}{2},1]:\cos\left(\frac{\nu(2\theta-\sin\theta)-\pi}{4}\right)\geq\frac{\sqrt{2}}{2}\right\}\right\arrowvert\geq C_{*}
\end{equation*}
with $C_{*}$ independent of $\nu$. Combining this with \eqref{PL4.8.1}, we get
\begin{equation*} \int_{0}^{\infty}\left(\varphi_n^{\alpha}(x)\right)^2(1+\arrowvert x\arrowvert)^{-\beta}x^{2\alpha+1}dx\geq\int_{1/2}^{1}\left(\varphi_n^{\alpha}(x)\right)^2(1+\arrowvert x\arrowvert)^{-\beta}x^{2\alpha+1}dx\geq Cn^{-1/2}
\end{equation*} and hence the desired estimate \eqref{L4.8.2}. 
\end{proof}
\par We also need following estimate.
\begin{lemma}\label{L4.9}
For every $\alpha>-1$, let $\beta\geq 0$. Then, for all $f\in L^2(\mathbb{R}_+^{d},d\mu_{\alpha}(x))$, we have the estimate
\begin{equation}\label{L4.9.1}
n^{\beta/4}\lVert P_nf\rVert_{2}\leq C \lVert P_nf\rVert_{L^2(\mathbb{R}_+^{d},(1+\arrowvert x\arrowvert)^{\beta}d\mu_{\alpha}(x))}.
\end{equation}
\end{lemma}
Since the proof of Lemma \ref{L4.9} follows directly from Lemma \ref{L4.8} and its proof method is similar to that of Lemma \ref{Trace}, we omit the details here. For a thorough discussion, the reader may refer to \cite[Lemma 4.10]{CD}. We now proceed with the proof of Proposition \ref{Sp2} using the argument as that in \cite{CD}.\\
\par \textbf{Proof of Proposition \ref{Sp2}.}
Since we are assuming that \eqref{Sp2.1} holds, by duality we have the equivalent estimate
\begin{equation*}
\sup_{R>0}\lVert S_{R}^{\lambda}(L_{\alpha})f\rVert_{L^2(\mathbb{R}_{+}^d,(1+\arrowvert x\arrowvert)^{\beta}d\mu_{\alpha})\rightarrow L^2(\mathbb{R}_{+}^d,(1+\arrowvert x\arrowvert)^{\beta}d\mu_{\alpha})}<\infty.
\end{equation*}
Combining this and \eqref{P} we obtain for function $F$ with supp\,$F\subset[0,\infty)$ that
\begin{equation*}
\begin{aligned}
\lVert F(L_{\alpha})f\rVert_{L^{2}(\mathbb{R}_{+}^d,(1+\arrowvert x\arrowvert)^{\beta}d\mu_{\alpha}(x))}&\leq C \sup_{R>0}\lVert S_{R}^{\lambda}(L_{\alpha})f\rVert_{L^{2}(\mathbb{R}_{+}^d,(1+\arrowvert x\arrowvert)^{\beta}d\mu_{\alpha}(x))}\int_0^{\infty}\arrowvert F^{\lambda+1}(s)s^{\lambda}ds\\
&\leq C\lVert f\rVert_{L^{2}(\mathbb{R}_{+}^d,(1+\arrowvert x\arrowvert)^{\beta}d\mu_{\alpha})}\int_0^{\infty}\arrowvert F^{\lambda+1}(s)s^{\lambda}ds.
\end{aligned}
\end{equation*}
We take $F(t)=\eta(t-(4n+2\arrowvert\alpha\arrowvert_1+2d))$ as before in the above where $\eta$ is a non-negative smooth function with $\eta(0)=1$ and supp\,$\eta\subset[-1,1]$. Then, since $\int_0^{\infty}\arrowvert \eta^{\lambda+1}(s-(4n+2\arrowvert\alpha\arrowvert_1+2d))s^{\lambda}ds\sim n^{\lambda}$, and $\eta(L_{\alpha}-(4n+2\arrowvert\alpha\arrowvert_1+2d))f=P_nf$, it follows that
\begin{equation}\label{P4.28}
\lVert P_n f\rVert_{L^2(\mathbb{R}_{+}^d,(1+\arrowvert x\arrowvert)^{\beta}d\mu_{\alpha}(x))}\leq Cn^{\lambda}\lVert f\rVert_{L^2(\mathbb{R}_{+}^d,(1+\arrowvert x\arrowvert)^{\beta}d\mu_{\alpha}(x))}.
\end{equation}
We now consider specific functions $g_n$, $G_n$ which are given by
\begin{equation*}
g_n(x)=\varphi_n^{\alpha_1}(x_1)\varphi_0^{\alpha_2}(x_2)\cdot\cdot\cdot\varphi_0^{\alpha_d}(x_d),\quad G_n(x)=g_n(x)(1+\arrowvert x\arrowvert)^{-\beta},
\end{equation*}
and claim that
\begin{equation}\label{P4.29}
\lVert P_n G_n\rVert_{L^2(\mathbb{R}_{+}^d,(1+\arrowvert x\arrowvert)^{\beta}d\mu_{\alpha}(x))}\leq C n^{\lambda}\min\{n^{\beta/4},n^{1/4}\}\lVert P_n G_n\rVert_{L^2(\mathbb{R}_{+}^d,d\mu_{\alpha}(x))}
\end{equation}
with $C$ independent of $n$. As a matter of fact, since $\lVert g_n(x)\rVert_2=1$, we have $\lVert P_n G_n\rVert_{2}\geq\langle P_n G_n,g_n\rangle_{\alpha}=\langle G_n,P_ng_n\rangle_{\alpha}$. Thus, noting that $P_ng_n=g_n$ from our choice of $g_n$ and $g_n(1+\arrowvert x\arrowvert)^{-\beta/2}=G_n(1+\arrowvert x\arrowvert)^{\beta/2}$, we obtain that
\begin{equation}\label{P4.30}
\lVert P_n G_n\rVert_{2}\geq \langle G_n,g_n\rangle_{\alpha}= \lVert(1+\arrowvert x\arrowvert)^{\beta/2} G_n\rVert_{2}\lVert(1+\arrowvert x\arrowvert)^{-\beta/2} g_n\rVert_{2}.
\end{equation}
From our choice of $g_n$, we have
\begin{equation*}
\lVert g_n\rVert_{L^2(\mathbb{R}_{+}^d,(1+\arrowvert x\arrowvert)^{-\beta}d\mu_{\alpha}(x))}^2\geq\int_0^{\infty}\arrowvert\varphi_n^{\alpha_1}(x_1)\arrowvert^2(1+\arrowvert x_1\arrowvert)^{-\beta}x_1^{2\alpha_1+1}dx_1\prod_{i=2}^{d}\int_0^{1}\arrowvert\varphi_0^{\alpha_i}(t)\arrowvert^2t^{2\alpha_i+1}dt,
\end{equation*}
by the estimate \eqref{L4.8.2}, it follows that
$$\lVert g_n\rVert_{L^2(\mathbb{R}_{+}^d,(1+\arrowvert x\arrowvert)^{-\beta}d\mu_{\alpha}(x))}^2\geq C\max\{n^{-\beta/2},n^{-1/2}\}.$$
Combining this with \eqref{P4.30} yields 
\begin{equation*}
\lVert P_n G_n\rVert_{2}\geq C\max\{n^{-\beta/2},n^{-1/2}\} \lVert G_n\rVert_{L^{2}(\mathbb{R}_{+}^d,(1+\arrowvert x\arrowvert)^{\beta}d\mu_{\alpha})}.
\end{equation*}
Using \eqref{P4.28}, we also have $n^{\lambda}\lVert G_n\rVert_{L^{2}(\mathbb{R}_{+}^d,(1+\arrowvert x\arrowvert)^{\beta}d\mu_{\alpha})}\geq C\lVert P_nG_n\rVert_{L^{2}(\mathbb{R}_{+}^d,(1+\arrowvert x\arrowvert)^{\beta}d\mu_{\alpha})}$. Thus we have the estimate \eqref{P4.29}. Now, we apply the estimate \eqref{L4.9.1} to the function $G_n$ and combine the consequent estimate with \eqref{P4.29} to get 
\begin{equation}\label{P4.31}
n^{\beta/4}\lVert P_n G_n\rVert_2\leq Cn^{\lambda}\min\{n^{\beta/4},n^{1/4}\}\lVert P_n G_n\rVert_{2}
\end{equation}
with $C$ independent of $n$. Since $\lVert P_n G_n\rVert_{2}\neq 0$, \eqref{P4.31} implies $n^{\beta/4}\leq Cn^{\lambda}\min\{n^{\beta/4},n^{1/4}\}$ with $C$ independent of $n$. Letting $n\rightarrow\infty$ gives $\lambda\geq \max\{(\beta-1)/4,0\}$ as desired.$\hfill{\Box}$
\begin{appendix}
\section{Proof of Lemma \ref{SL}}\label{appendix}
In this appendix, we provide a precise proof of Lemma \ref{SL}, and the argument follows the same reasoning as the proof of \cite[Lemma 3.1]{CD}. For this purpose, we recall the  Lemma \ref{SL} as follows.
\begin{lemma}
Let $A_{\beta,d}^{\varepsilon}(\delta)$ be given by \eqref{S4}. Then, for all $0<\delta<1/2$ and $0<\varepsilon\leq1/2$, we have the following estimates:
\begin{equation}\label{Alow}
\int_{\mathbb{R}_{+}^{d}}\arrowvert\mathfrak{S}_{\delta}^{l}f(x)\arrowvert^2(1+\arrowvert x\arrowvert)^{-\beta}d\mu_{\alpha}(x)\leq C\delta A_{\beta,d}^{\varepsilon}(\delta)\int_{\mathbb{R}_{+}^{d}}\arrowvert f(x)\arrowvert^2(1+\arrowvert x\arrowvert)^{-\beta}d\mu_{\alpha}(x),
\end{equation}
\begin{equation}\label{Ahigh}
\int_{\mathbb{R}_{+}^{d}}\arrowvert\mathfrak{S}_{\delta}^{h}f(x)\arrowvert^2(1+\arrowvert x\arrowvert)^{-\beta}d\mu_{\alpha}(x)\leq C\delta A_{\beta,d}^{\varepsilon}(\delta)\int_{\mathbb{R}_{+}^{d}}\arrowvert f(x)\arrowvert^2(1+\arrowvert x\arrowvert)^{-\beta}d\mu_{\alpha}(x).
\end{equation}
\end{lemma}
\subsection{Proof of \eqref{Alow}: low frequency part}
\begin{proof}
We take the second case for example, that is, we assume $-d<\arrowvert \alpha\arrowvert_1<0$. We start with Littlewood-Paley decomposition associated with the operator $L_{\alpha}$. Fix a function $\psi\in C^{\infty}$ in $(\frac{1}{2},2)$ such that $\sum_{-\infty}^{\infty}\psi(2^{-k}x)=1$ on $\mathbb{R}\backslash\{0\}$. By the spectral theory, we obtain
\begin{equation}\label{LPC}
\sum_{k}\psi_{k}(\sqrt{L_{\alpha}})f:=\sum_{k}\psi (2^{-k}\sqrt{L_{\alpha}})f=f
\end{equation}
for any $f\in L^2(\mathbb{R}_{+}^{d},d\mu_{\alpha}(x))$. Using \eqref{LPC}, we get
\begin{equation}\label{A2}
\arrowvert \mathfrak{S}_{\delta}^{l}f(x)\arrowvert ^2\leq C \sum_{[\log_2\sqrt{\arrowvert \alpha\arrowvert_1+d}]\leq k\leq 1-\log_2\sqrt{\delta}}\int_{2^{k-1}}^{2^{k+2}}\arrowvert \phi\left(\delta^{-1}\left(1-\frac{L_{\alpha}}{t^2}\right)\right)\psi_{k}(\sqrt{L_{\alpha}})f(x)\arrowvert^2\frac{dt}{t}
\end{equation}
for $f\in L^2(\mathbb{R}_{+}^{d},d\mu_{\alpha}(x))\cap L^2(\mathbb{R}_{+}^{d},(1+\arrowvert x\arrowvert)^{-\beta}d\mu_{\alpha}(x))$. To exploit disjointness of the spectral support $\arrowvert \phi\left(\delta^{-1}\left(1-\frac{L_{\alpha}}{t^2}\right)\right)$, we make additional decomposition in $t$. For $k\in \mathbb{Z}$ and $i=0,1,\cdot\cdot\cdot,i_o=[8/\delta]+1$ we set
\begin{equation}\label{A3}
I_i=[2^{k-1}+,i2^{k-1}\delta,2^{k-1}+(i+1)2^{k-1}\delta],
\end{equation}
so that $[2^{k-1},2^{k+2}]\subset \cup_{i=0}^{i_0}I_i$. Define $\eta_i$ adapted to the interval $I_i$ by setting
\begin{equation}\label{A4}
\eta_i(s)=\eta\left(i+\frac{2^{k-1}-s}{2^{k-1}\delta}\right),
\end{equation}
where $\eta\in C_{c}^{\infty}(-1,1)$ and $\sum_{i\in\mathbb{Z}}\eta(\cdot-i)=1$. For simplicity we also set
$$\phi_{\delta}(s):=\phi(\delta^{-1}(1-s^2)).$$
Then for $t\in I_i$, $\phi_{\delta}(s/t)\eta_i'(s)\neq0$ only if $i-i\delta-3\leq i'\leq i+i\delta+3$, see \cite[p.20]{CD}. Thus we can deduce that
\begin{equation}\label{Bclaim}
\arrowvert \mathfrak{S}_{\delta}^{l}f(x)\arrowvert ^2\leq C \sum_{[\log_2\sqrt{\arrowvert \alpha\arrowvert+d}]\leq k\leq 1-\log_2\sqrt{\delta}}\sum_{i}\sum_{i'=i-10}^{i+10}\int_{I_i}\arrowvert \phi_{\delta}(t^{-1}\sqrt{L_{\alpha}})\eta_{i'}(\sqrt{L_{\alpha}})\psi_{k}(\sqrt{L_{\alpha}})f(x)\arrowvert^2\frac{dt}{t}.
\end{equation}
Now, we claim that, for $2^{[\log_2\sqrt{\arrowvert \alpha\arrowvert_1+d}]-1}\leq t\leq \delta^{-1/2}$,
\begin{equation}\label{Claim}
\int_{\mathbb{R}_{+}^{d}}\arrowvert \phi_{\delta}(t^{-1}\sqrt{L_{\alpha}}) g(x)\arrowvert^2(1+\arrowvert x\arrowvert)^{-\beta}d\mu_{\alpha}(x)\leq C A_{\beta,d}^{\varepsilon}(\delta)\int_{\mathbb{R}_{+}^{d}}\arrowvert (1+L_{\alpha})^{-\beta/4} g(x)\arrowvert^2d\mu_{\alpha}(x).
\end{equation}
To prove this claim, let us consider the equivalent estimate
\begin{equation}\label{Claim1}
\int_{\mathbb{R}_{+}^{d}}\arrowvert \phi_{\delta}(t^{-1}\sqrt{L_{\alpha}})(1+L_{\alpha})^{\beta/4} g(x)\arrowvert^2(1+\arrowvert x\arrowvert)^{-\beta}d\mu_{\alpha}(x)\leq C A_{\beta,d}^{\varepsilon}(\delta)\int_{\mathbb{R}_{+}^{d}}\arrowvert g(x)\arrowvert^2d\mu_{\alpha}(x).
\end{equation}
We first show the estimate for the case $2(\arrowvert \alpha\arrowvert_1 +d)>1$. Let $N=8[t]+1$. Note that supp\,$\phi_{\delta}(\cdot/t)\subset[N/4,N]$. Then by estimate \eqref{ETrace1} in Lemma \ref{ETrace}, we obtain
\begin{equation*}
\begin{aligned}
&\int_{\mathbb{R}_{+}^{d}}\arrowvert \phi_{\delta}(t^{-1}\sqrt{L_{\alpha}})(1+L_{\alpha})^{\beta/4} g(x)\arrowvert^2(1+\arrowvert x\arrowvert)^{-\beta}d\mu_{\alpha}(x)\\
&\leq CN\left\lVert \phi_{\delta}(t^{-1} Nx)(1+N^2x^2)^{\beta/4}\right\rVert_{N^2,2}^2\int_{\mathbb{R}_{+}^{d}}\arrowvert g(x)\arrowvert^2d\mu_{\alpha}(x).
\end{aligned}
\end{equation*}
Estimate (3.12) in \cite{CD} tells us that
$$\left\lVert \phi_{\delta}(t^{-1} Nx)(1+N^2x^2)^{\beta/4}\right\rVert_{N^2,2}^2 \leq CN^{\beta-2}.$$
Thus, noting  $2^{[\log_2\sqrt{\arrowvert \alpha\arrowvert+d}]-1}\leq t\leq \delta^{-1/2}$ and $\beta>1$, we obtain
\begin{equation*}
\int_{\mathbb{R}_{+}^{d}}\arrowvert \phi_{\delta}(t^{-1}\sqrt{L_{\alpha}})(1+L_{\alpha})^{\beta/4} g(x)\arrowvert^2(1+\arrowvert x\arrowvert)^{-\beta}d\mu_{\alpha}(x)\leq C\delta^{1/2-\beta/2}\int_{\mathbb{R}_{+}^{d}}\arrowvert g(x)\arrowvert^2d\mu_{\alpha}(x),
\end{equation*}
which gives \eqref{Claim1} and thus \eqref{Claim} in the case $2(\arrowvert \alpha\arrowvert_1 +d)>1$. If $2(\arrowvert \alpha\arrowvert_1 +d)\leq1$, for any $\varepsilon>0$, then by \eqref{ETrace2} in Lemma \ref{ETrace}, we have
\begin{equation}
\begin{aligned}
&\int_{\mathbb{R}_{+}^{d}}\arrowvert \phi_{\delta}(t^{-1}\sqrt{L_{\alpha}})(1+L_{\alpha})^{\beta/4} g(x)\arrowvert^2(1+\arrowvert x\arrowvert)^{-\beta}d\mu_{\alpha}(x)\\
&\leq C_{\varepsilon}N^{\frac{\beta}{1+\varepsilon}}\left\lVert \phi_{\delta}(t^{-1} Nx)(1+N^2x^2)^{\beta/4}\right\rVert_{N^2,\frac{2(1+\varepsilon)}{\beta}}^2\int_{\mathbb{R}_{+}^{d}}\arrowvert g(x)\arrowvert^2d\mu_{\alpha}(x).
\end{aligned}
\end{equation}
As before, we have (see \cite[p.22]{CD})
$$\left\lVert \phi_{\delta}(t^{-1} Nx)(1+N^2x^2)^{\beta/4}\right\rVert_{N^2,\frac{2(1+\varepsilon)}{\beta}}^2\leq CN^{\frac{\beta(\varepsilon-1)}{1+\varepsilon}}.$$
So it follows that
\begin{equation*}
\int_{\mathbb{R}_{+}^{d}}\arrowvert \phi_{\delta}(t^{-1}\sqrt{L_{\alpha}})(1+L_{\alpha})^{\beta/4} g(x)\arrowvert^2(1+\arrowvert x\arrowvert)^{-\beta}d\mu_{\alpha}(x)\leq C \delta^{-\frac{\beta}{2(1+\varepsilon)}}\int_{\mathbb{R}_{+}^{d}}\arrowvert g(x)\arrowvert^2d\mu_{\alpha}(x).
\end{equation*}
This gives \eqref{Claim} with $2(\arrowvert \alpha\arrowvert_1 +d)\leq1$. Now, combining \eqref{Claim} and \eqref{Bclaim}, we see that $\int_{\mathbb{R}_{+}^{d}}\arrowvert \mathfrak{S}_{\delta}^{l}f(x)\arrowvert ^2 (1+\arrowvert x\arrowvert)^{-\beta}d\mu_{\alpha}(x)$ is bounded by
\begin{equation*}
 CA_{\beta,d}^{\varepsilon}(\delta)\sum_{[\log_2\sqrt{\arrowvert \alpha\arrowvert+d}]\leq k\leq 1-\log_2\sqrt{\delta}}\sum_{i}\sum_{i'=i-10}^{i+10}\int_{I_i}\arrowvert \eta_{i'}(\sqrt{L_{\alpha}})\psi_{k}(\sqrt{L_{\alpha}})(1+L_{\alpha})^{-\beta/4} f(x)\arrowvert^2\frac{dt}{t}.
\end{equation*}
Since the length of interval $I_i$ is compareble to $2^{k-1}\delta$, taking integration in $t$ and using disjointness of the spectral supports, we obtain
\begin{equation*}
\int_{\mathbb{R}_{+}^{d}}\arrowvert \mathfrak{S}_{\delta}^{l}f(x)\arrowvert ^2 (1+\arrowvert x\arrowvert)^{-\beta}d\mu_{\alpha}(x)\leq C\delta A_{\beta,d}^{\varepsilon}(\delta)\int_{\mathbb{R}_{+}^{d}}\arrowvert(1+L_{\alpha})^{-\beta/4}f(x)\arrowvert^2d\mu_{\alpha}(x).
\end{equation*}
This, being combined with Lemma \ref{weighted} , yields the desired estimate \eqref{Alow}.
\end{proof}
\subsection{Proof of \eqref{Ahigh}: high frequency part}
For any even function $F$ with $\hat{F}\in L^1(\mathbb{R})$ and supp\,$\hat{F}\subseteq[-r,r]$ we can deduce from Lemma \ref{Finit} that (see also \cite[3.14, p.23]{CD})
\begin{equation}\label{B2}
supp\,K_{F(t\sqrt{L_{\alpha}})}(x,y)\subset D(t^{-1}r).
\end{equation} 
\par Fixing an even function $\vartheta\in C_{c}^{\infty}$ which is identically one on $\{\arrowvert s\arrowvert\leq 1\}$ and supported on $\{\arrowvert s\arrowvert\leq 2\}$, let us set $j_0=[-\log_2\delta]-1$ and $\zeta_{j_o}(s):=\vartheta(2^{-j_0}s)$ and $\zeta_{j}(s):=\vartheta(2^{-j}s)-\vartheta(2^{-j+1}s)$ for $j>j_0$. Then, we have $\sum_{j\geq j_0}\zeta_{j}(s)\equiv1,\forall s>0$.
Recalling that $\phi_{\delta}(s)=\phi(\delta^{-1}(1-s^2))$, for $j\geq j_0$, we set
\begin{equation}\label{B3}
\phi_{\delta,j}(s)=\frac{1}{2\pi}\int_{0}^{\infty}\zeta_j(u)\hat{\phi_{\delta}}(u)\cos(su)du.
\end{equation}
It can be verified that
\begin{equation}\label{B4}
\arrowvert\phi_{\delta,j}(s)\arrowvert\leq\left\{\begin{array}{lr}C_N2^{(j_0-j)N},\qquad\qquad\qquad \arrowvert s\arrowvert \in [1-2\delta,1+2\delta],\\
C_N2^{j-j_0}(1+2^j\arrowvert s-1\arrowvert)^{-N},\quad\quad otherwise,
\end{array}
\right.
\end{equation}
for any $N$ and all $j\geq j_0$ (see \cite[3.16, p.23]{CD}). By the Fourier inversion formula, we obtain
\begin{equation}\label{B5}
\phi_{\delta}(s)=\phi(\delta^{-1}(1-s^2))=\sum_{j\geq j_0}\phi_{\delta,j}(s),\quad s>0.
\end{equation}
By \eqref{B2}, we have 
\begin{equation}\label{B6}
supp\,K_{\phi_{\delta,j}(\sqrt{L_{\alpha}}/t)}(x,y)\subset D(t^{-1}2^{j+1}):=\{(x,y)\in \mathbb{R}_{+}^{d}\times\mathbb{R}_{+}^{d}:\arrowvert x-y\arrowvert\leq c_02^{j+1}/t\}.
\end{equation}
Now from \eqref{LPC}, it follows that
\begin{equation}\label{B7}
\arrowvert \mathfrak{S}_{\delta}^{l}f(x)\arrowvert ^2\leq C \sum_{k\geq 1-\log_2\sqrt{\delta}}\int_{2^{k-1}}^{2^{k+2}}\arrowvert \phi\left(\delta^{-1}\left(1-\frac{L_{\alpha}}{t^2}\right)\right)\psi_{k}(\sqrt{L_{\alpha}})f(x)\arrowvert^2\frac{dt}{t}.
\end{equation}
For $k\geq 1-\log_2\sqrt{\delta}$ and $j\geq j_0$, let us set
\begin{equation*}\label{B8}
E^{k,j}(t):=\left\langle\left\lvert\phi_{\delta,j}(t^{-1}\sqrt{L_{\alpha}})\psi_{k}(\sqrt{L_{\alpha}})f(x)\right\rvert^2,(1+\arrowvert x\arrowvert)^{-\beta} \right\rangle_{\alpha}. 
\end{equation*}
Using \eqref{B5}, \eqref{B7} and Minkowski's inequality, we have
\begin{equation}\label{B9}
\int_{\mathbb{R}_{+}^{d}}\arrowvert \mathfrak{S}_{\delta}^{l}f(x)\arrowvert ^2(1+\arrowvert x\arrowvert)^{-\beta}d\mu_{\alpha}(x)\leq C \sum_{k\geq 1-\log_2\sqrt{\delta}}\left(\sum_{j\geq j_0}\left(\int_{2^{k-1}}^{2^{k+2}}E^{k,j}(t)\frac{dt}{t}\right)^{1/2}\right)^2.
\end{equation}
In order to make use the localization property \eqref{B6}, we need to decompose $\mathbb{R}_{+}^{d}$ into disjoint cubes of side length $c_02^{j-k+2}$. For a given $k\in \mathbb{Z}$, $j\geq j_0$, and $m=(m_1,\cdot\cdot\cdot,m_d)\in \mathbb{Z}_{\geq 0}^{d}$, let us set 
\begin{equation*}\label{B10}
Q_{m}=\left[c_02^{j-k+2}(m_1-\frac{1}{2}), c_02^{j-k+2}(m_1+\frac{1}{2})\right)\times\cdot\cdot\cdot\times\left[c_02^{j-k+2}(m_d-\frac{1}{2}), c_02^{j-k+2}(m_d+\frac{1}{2})\right),
\end{equation*}
which are disjoint dyadic cubes centered at $x_m:=c_02^{j-k+2}m$ with side length $c_02^{j-k+2}$ and $\mathbb{R}_{+}^{d}=\cup_{m\in\mathbb{Z}_{\geq 0}^{d}}Q_m$. For each $m$, we define $\widetilde{Q_m}$ by setting
\begin{equation*}\label{B11}
\widetilde{Q_m}:=\bigcup_{m'\in\mathbb{Z}_{\geq 0}^{d}:dist(Q_{m'},Q_{m})\leq\sqrt{d}c_02^{j-k+3}}Q_{m'},
\end{equation*}
and denote
\begin{equation*}\label{B12}
M_0:=\{m\in\mathbb{Z}_{\geq 0}^{d}:Q_0\cap\widetilde{Q_m}\neq \emptyset\}.
\end{equation*}
For $t\in[2^{k-1},2^{k+2}]$ it follows by \eqref{B6} that $\chi_{Q_m}\phi_{\delta,j}(t^{-1}\sqrt{L_{\alpha}})\chi_{Q_{m'}}=0$ if $\widetilde{Q_m}\cap Q_{m'}=\emptyset$ for every $j,k$. Hence, it is clear that
\begin{equation*}\label{B13}
\phi_{\delta,j}(t^{-1}\sqrt{L_{\alpha}})\psi_{k}(\sqrt{L_{\alpha}})f=\sum_{m,m':dist(Q_m,Q_{m'})<t^{-1}2^{j+2}}\chi_{Q_m}\phi_{\delta,j}(t^{-1}\sqrt{L_{\alpha}})\chi_{Q_{m'}}\psi_{k}(\sqrt{L_{\alpha}}),
\end{equation*}
which gives 
\begin{equation}\label{B14}
E^{k,j}(t)\leq C\sum_{m}\left\langle\left\lvert\chi_{Q_m}\phi_{\delta,j}(t^{-1}\sqrt{L_{\alpha}})\chi_{\widetilde{Q_m}}\psi_{k}(\sqrt{L_{\alpha}})f(x)\right\rvert^2,(1+\arrowvert x\arrowvert)^{-\beta} \right\rangle_{\alpha}.
\end{equation}
To exploit orthogonality generated by the disjointness of spectral support, we further decompose $\phi_{\delta,j}$ which is not compactly supported. We choose an even function $\theta\in C_{c}^{\infty}(-4,4)$ such that $\theta(s)=1$ for $s\in(-2,2)$. Set
\begin{equation}\label{B15}
\varphi_{0,\delta}(s):=\theta(\delta^{-1}(1-s)),\quad\varphi_{\mathscr{l},\delta}(s):=\theta(2^{-\mathscr{l}}\delta^{-1}(1-s))-\theta(2^{-\mathscr{l}+1}\delta^{-1}(1-s))
\end{equation}
for all $\mathscr{l}\geq1$ such that $1=\sum_{\mathscr{l}=0}^{\infty}\varphi_{\mathscr{l},\delta}(s)$ and $\phi_{\delta,j}=\sum_{\mathscr{l}=0}^{\infty}(\varphi_{\mathscr{l},\delta}\phi_{\delta,j})(s)$ for all $s>0$. We put it into \eqref{B14} to write
\begin{equation}\label{B16}
\left(\int_{2^{k-1}}^{2^{k+2}}E^{k,j}(t)\frac{dt}{t}\right)^{1/2}\leq\sum_{\mathscr{l}=0}^{\infty}\left(\sum_{m}\int_{2^{k-1}}^{2^{k+2}}E_{m}^{k,j,\mathscr{l}}(t)\frac{dt}{t}\right)^{1/2},
\end{equation}
where 
\begin{equation*}
E_{m}^{k,j,\mathscr{l}}(t)=\left\langle\left\lvert\chi_{Q_m}(\varphi_{\mathscr{l},\delta}\phi_{\delta,j})(t^{-1}\sqrt{L_{\alpha}})\chi_{\widetilde{Q_m}}\psi_{k}(\sqrt{L_{\alpha}})f(x)\right\rvert^2,(1+\arrowvert x\arrowvert)^{-\beta} \right\rangle_{\alpha}.
\end{equation*}
Recalling that \eqref{A3} and \eqref{A4}, we observe that, for every $t\in I_i$, it is possible that $\varphi_{\mathscr{l},\delta}(s/t)\eta_{i'}(s) \neq 0 $ only when $i-2^{\mathscr{l}+6}\leq i'\leq i+2^{\mathscr{l}+6}$. Hence,
\begin{equation*}
(\varphi_{\mathscr{l},\delta}\phi_{\delta,j})(t^{-1}\sqrt{L_{\alpha}})=\sum_{i'=i-2^{\mathscr{l}+6}}^{i+2^{\mathscr{l}+6}}(\varphi_{\mathscr{l},\delta}\phi_{\delta,j})(t^{-1}\sqrt{L_{\alpha}})\eta_{i'}(\sqrt{L_{\alpha}}),\quad t\in I_i.
\end{equation*}
From this and Cauchy-Schwarz's inequality, we have
\begin{equation*}
E_{m}^{k,j,\mathscr{l}}(t)\leq C2^{\mathscr{l}}\sum_{i'=i-2^{\mathscr{l}+6}}^{i+2^{\mathscr{l}+6}}E_{m,i'}^{k,j,\mathscr{l}}(t)
\end{equation*}
for $t\in I_i$ where
\begin{equation*}
E_{m,i'}^{k,j,\mathscr{l}}(t)=\left\langle\left\lvert\chi_{Q_m}(\varphi_{\mathscr{l},\delta}\phi_{\delta,j})(t^{-1}\sqrt{L_{\alpha}})\eta_{i'}(\sqrt{L_{\alpha}})\left[\chi_{\widetilde{Q_m}}\psi_{k}(\sqrt{L_{\alpha}})f\right](x)\right\rvert^2,(1+\arrowvert x\arrowvert)^{-\beta} \right\rangle_{\alpha}.
\end{equation*}
Combining this with \eqref{B16}, we get
\begin{equation}\label{B17}
\left(\int_{2^{k-1}}^{2^{k+2}}E^{k,j}(t)\frac{dt}{t}\right)^{1/2}\leq C\sum_{\mathscr{l}=0}^{\infty}2^{\mathscr{l}/2}\left(\sum_{m}\sum_{i}\int_{I_i}\sum_{i'=i-2^{\mathscr{l}+6}}^{i+2^{\mathscr{l}+6}}E_{m,i'}^{k,j,\mathscr{l}}(t)\frac{dt}{t}\right)^{1/2}.
\end{equation}
To continue, we distinguish two cases: $j>k$; and $j\leq k$. The first case is more involved and we need to distinguish several sub-cases which we separately handle.\\
\par\textbf{Case $j>k$} From the above inequality  \eqref{B17} we have
\begin{equation*}
\left(\int_{2^{k-1}}^{2^{k+2}}E^{k,j}(t)\frac{dt}{t}\right)^{1/2}\leq I_1(j,k)+I_2(j,k)+I_3(j,k),
\end{equation*}
where 
\begin{equation}\label{B18}
I_1(j,k):=\sum_{\mathscr{l}=0}^{[-\log_2\delta]-3}2^{\mathscr{l}/2}\left(\sum_{m\in M_{0}}\sum_{i}\int_{I_i}\sum_{i'=i-2^{\mathscr{l}+6}}^{i+2^{\mathscr{l}+6}}E_{m,i'}^{k,j,\mathscr{l}}(t)\frac{dt}{t}\right)^{1/2},
\end{equation}
\begin{equation}\label{B19}
I_2(j,k):=\sum_{\mathscr{l}=[-\log_2\delta]-2}^{\infty}2^{\mathscr{l}/2}\left(\sum_{m\in M_{0}}\sum_{i}\int_{I_i}\sum_{i'=i-2^{\mathscr{l}+6}}^{i+2^{\mathscr{l}+6}}E_{m,i'}^{k,j,\mathscr{l}}(t)\frac{dt}{t}\right)^{1/2},
\end{equation}
\begin{equation}\label{B20}
I_3(j,k):=\sum_{\mathscr{l}=0}^{\infty}2^{\mathscr{l}/2}\left(\sum_{m\notin M_{0}}\sum_{i}\int_{I_i}\sum_{i'=i-2^{\mathscr{l}+6}}^{i+2^{\mathscr{l}+6}}E_{m,i'}^{k,j,\mathscr{l}}(t)\frac{dt}{t}\right)^{1/2}.
\end{equation}
We will now estimate $I_1(j,k)$, $I_2(j,k)$, and $I_3(j,k)$ one by one.\\
\\
\textit{Estimate of the term  $I_1(j,k)$.}   We claim that, for any $N>0$,
\begin{equation}\label{B21}
I_1(j,k)\leq C_N 2^{(j_0-j)N}\left(\delta A_{\beta,d}^{\varepsilon}(\delta)\right)^{1/2}\left(\int_{\mathbb{R}_{+}^{d}}\arrowvert \psi_k(\sqrt{L_{\alpha}})f(x)\arrowvert ^2(1+\arrowvert x\arrowvert)^{-\beta}d\mu_{\alpha}(x)\right)^{1/2}
\end{equation}
where $A_{\beta,d}^{\varepsilon}(\delta)$ is defined in \eqref{S4}. Let us consider the case $2(\arrowvert\alpha\arrowvert_1+d)>1$. For \eqref{B21}, it suffices to show
\begin{equation}\label{B22}
E_{m,i'}^{k,j,\mathscr{l}}(t)\leq C_N2^{-\mathscr{l}N}2^{(j_0-j)N}2^{k}\delta \int_{\mathbb{R}_{+}^{d}}\arrowvert \eta_{i'}(\sqrt{L_{\alpha}})[ \chi_{\widetilde{Q_m}}\psi_k(\sqrt{L_{\alpha}})f](x)\arrowvert ^2d\mu_{\alpha}(x)
\end{equation}
for any $N>0$ while $t\in I_i$ being fixed and $i-2^{\mathscr{l}+6}\leq i'\leq i+2^{\mathscr{l}+6}$. Indeed, since the support of $\eta_{i}$ are boundedly overlapping, \eqref{B22} gives that
\begin{equation*}
\sum_{i}\int_{I_i}\sum_{i'=i-2^{\mathscr{l}+6}}^{i+2^{\mathscr{l}+6}}E_{m,i'}^{k,j,\mathscr{l}}(t)\frac{dt}{t}\leq C_N2^{-\mathscr{l}(N-1)}2^{(j_0-j)N}2^{k}\delta^2\lVert \chi_{\widetilde{Q_m}}\psi_k(\sqrt{L_{\alpha}})f\rVert_2^2.
\end{equation*}
Now we take summation over $\mathscr{l}$ and $m\in M_{0}$ in \eqref{B18} to get
\begin{equation*}
I_1(j,k)\leq C_N 2^{j_0\beta/2}2^{(j_0-j)(N-\beta)/2}2^{k(1-\beta)/2}\delta\left(2^{(k-j)\beta}\sum_{m\in M_0}\lVert \chi_{\widetilde{Q_m}}\psi_k(\sqrt{L_{\alpha}})f\rVert_2^2\right)^{1/2}.
\end{equation*}
Since $j>k$ and $m\in M_0$, we note that $(1+\arrowvert x\arrowvert)^{\beta}\leq C2^{(j-k)\beta}$ if $x\in \widetilde {Q_m}$. It follows that
\begin{equation}\label{BB}
2^{(k-j)\beta}\sum_{m\in M_0}\lVert \chi_{\widetilde{Q_m}}\psi_k(\sqrt{L_{\alpha}})f\rVert_2^2\leq C\int_{\mathbb{R}_{+}^{d}}\arrowvert \psi_k(\sqrt{L_{\alpha}})f(x)\arrowvert ^2(1+\arrowvert x\arrowvert)^{-\beta}d\mu_{\alpha}(x).
\end{equation}
Noting that $j_0=[-\log_2\delta]-1$ and $k\geq [-\frac{1}{2}\log_2\delta]$, we obtain
\begin{equation*}
I_1(j,k)\leq C_N 2^{(j_0-j)(N-\beta)/2}\delta^{3/4-\beta/4}\left(\int_{\mathbb{R}_{+}^{d}}\arrowvert \psi_k(\sqrt{L_{\alpha}})f(x)\arrowvert ^2(1+\arrowvert x\arrowvert)^{-\beta}d\mu_{\alpha}(x)\right)^{1/2},
\end{equation*}
which gives \eqref{B21} since $N>0$ is arbitrary. We now proceed to prove \eqref{B22}. Note that supp\,$\varphi_{\mathscr{l},\delta}\subseteq (1-2^{\mathscr{l}+2}\delta,1+2^{\mathscr{l}+2}\delta)$, and so supp\,$(\varphi_{\mathscr{l},\delta}\phi_{\delta,j})(\cdot/t)\subset [t(1-2^{\mathscr{l}+2}\delta), t(2^{\mathscr{l}+2}\delta)]$. Thus, setting $R=[t(1+2^{\mathscr{l}+2}\delta)]$, by \eqref{ETrace1} we get
\begin{equation}\label{B23}
E_{m,i'}^{k,j,\mathscr{l}}(t)\leq R\lVert (\varphi_{\mathscr{l},\delta}\phi_{\delta,j})(R\cdot/t)\rVert_{R^2,2}^2\int_{\mathbb{R}_{+}^{d}}\arrowvert \eta_{i'}(\sqrt{L_{\alpha}})[ \chi_{\widetilde{Q_m}}\psi_k(\sqrt{L_{\alpha}})f](x)\arrowvert ^2d\mu_{\alpha}(x)
\end{equation}
for $0\leq \mathscr{l}\leq [-\log_2\delta]-3$. By formula (3.34) in \cite{CD}, we have
\begin{equation}\label{BBB}
\lVert (\varphi_{\mathscr{l},\delta}\phi_{\delta,j})(R\cdot/t)\rVert_{R^2,2}\leq\lVert (\varphi_{\mathscr{l},\delta}\phi_{\delta,j})((1+2^{\mathscr{l}+2}\delta)\cdot)\rVert_{\infty}(2^{\mathscr{l}+3}\delta)^{1/2}\leq C_N2^{(j_0-j)N}2^{-\mathscr{l}N}(2^{\mathscr{l}+3}\delta)^{1/2}
\end{equation}
for any $N>0$ and $\mathscr{l}\geq 0$. Since $R\sim 2^{k}$, combining this with \eqref{B23}, we get the desired \eqref{B22}.
\par For the case $2(\arrowvert \alpha\arrowvert_1+d)\leq 1$, it can be handled in the same manner as before, the only difference is that we use \eqref{ETrace2} instead of \eqref{ETrace1}, so we omit those steps.\\
\\
\textit{Estimate of the term  $I_2(j,k)$.} We can deduce as before that
\begin{equation*}
E_{m,i'}^{k,j,\mathscr{l}}(t)\leq \lVert (\varphi_{\mathscr{l},\delta}\phi_{\delta,j})(\cdot/t)\rVert_{\infty}^2\int_{\mathbb{R}_{+}^{d}}\arrowvert \eta_{i'}(\sqrt{L_{\alpha}})[ \chi_{\widetilde{Q_m}}\psi_k(\sqrt{L_{\alpha}})f](x)\arrowvert ^2d\mu_{\alpha}(x).
\end{equation*}
From the definition of $\varphi_{\mathscr{l},\delta}$ and \eqref{B4} we have
\begin{equation*}
\lVert (\varphi_{\mathscr{l},\delta}\phi_{\delta,j})(\cdot/t)\rVert_{\infty}\leq C_N2^{j-j_0}(2^{j+\mathscr{l}}\delta)^{-N},\quad \mathscr{l}\geq[-\log_2\delta]-2.
\end{equation*}
Thus, it follows that
\begin{equation*}
E_{m,i'}^{k,j,\mathscr{l}}(t)\leq C_N2^{2(j-j_0)}(2^{j+\mathscr{l}}\delta)^{-2N}\int_{\mathbb{R}_{+}^{d}}\arrowvert \eta_{i'}(\sqrt{L_{\alpha}})[ \chi_{\widetilde{Q_m}}\psi_k(\sqrt{L_{\alpha}})f](x)\arrowvert ^2d\mu_{\alpha}(x).
\end{equation*}
After putting this in \eqref{B19}, we take summation over $m\in M_{0}$ to obtain
\begin{equation*}
I_2(j,k)\leq C_N\delta^{1/2-N}2^{(j-j_0)}2^{-N_j}2^{(j-k)\alpha/2}\sum_{\mathscr{l}=[-\log_2\delta]-2}^{\infty}2^{-\mathscr{l}(N-2)}\left(2^{(k-j)\beta}\sum_{m\in M_0}\lVert \chi_{\widetilde{Q_m}}\psi_k(\sqrt{L_{\alpha}})f\rVert_2^2\right)^{1/2}.
\end{equation*}
Since $j>k$, $j_0=[-\log_2\delta]-1$ and $k\geq[-\frac{1}{2}\log_2\delta]$, using \eqref{BB} and taking sum over $\mathscr{l}$ we obtain
\begin{equation}\label{B24}
I_2(j,k)\leq C_N2^{(j_0-j)N}\delta^N \left(\int_{\mathbb{R}_{+}^{d}}\arrowvert \psi_k(\sqrt{L_{\alpha}})f(x)\arrowvert ^2(1+\arrowvert x\arrowvert)^{-\beta}d\mu_{\alpha}(x)\right)^{1/2}
\end{equation}
for any $N>0$.\\
\\
\textit{Estimate of the term  $I_3(j,k)$.} We now prove that
\begin{equation}\label{B25}
I_3(j,k)\leq C_N2^{(j_0-j)N}\delta^{1/2} \left(\int_{\mathbb{R}_{+}^{d}}\arrowvert \psi_k(\sqrt{L_{\alpha}})f(x)\arrowvert ^2(1+\arrowvert x\arrowvert)^{-\beta}d\mu_{\alpha}(x)\right)^{1/2}.
\end{equation}
Recall that $x_m=2^{j-k+2}m$ is the center of the cube $Q_m$. We begin with making an observation that
\begin{equation}\label{B26}
C^{-1}(1+\arrowvert x_m\arrowvert)\leq 1+\arrowvert x\arrowvert\leq C(1+\arrowvert x_m\arrowvert),\quad x\in Q_m
\end{equation}
provided that $m\notin M_{0}$. By this observation, we can deduce that 
\begin{equation*}
E_{m,i'}^{k,j,\mathscr{l}}(t)\leq (1+\arrowvert x_m\arrowvert)^{-\beta}\lVert (\varphi_{\mathscr{l},\delta}\phi_{\delta,j})(t^{-1}\sqrt{L_{\alpha}})\rVert_{2\rightarrow 2}^{2}\lVert  \eta_{i'}(\sqrt{L_{\alpha}})[ \chi_{\widetilde{Q_m}}\psi_k(\sqrt{L_{\alpha}})f]\rVert_2^2.
\end{equation*}
Since $\lVert (\varphi_{\mathscr{l},\delta}\phi_{\delta,j})(t^{-1}\sqrt{L_{\alpha}})\rVert_{2\rightarrow 2}\leq \lVert (\varphi_{\mathscr{l},\delta}\phi_{\delta,j})\rVert_{\infty}$ ,it follows from \eqref{BBB} that we have
\begin{equation*}
E_{m,i'}^{k,j,\mathscr{l}}(t)\leq (1+\arrowvert x_m\arrowvert)^{-\beta}2^{2(j_0-j)N}2^{-2\mathscr{l}N}\lVert  \eta_{i'}(\sqrt{L_{\alpha}})[ \chi_{\widetilde{Q_m}}\psi_k(\sqrt{L_{\alpha}})f]\rVert_2^2.
\end{equation*}
Using this and disjointness of the spectral supports, we obtain
\begin{equation*}
\begin{aligned}
\sum_{m\notin M_{0}}\sum_{i}\int_{I_i}\sum_{i'=i-2^{\mathscr{l}+6}}^{i+2^{\mathscr{l}+6}}E_{m,i'}^{k,j,\mathscr{l}}(t)\frac{dt}{t}&\leq C2^{2(j_0-j)N}2^{-2\mathscr{l}N}\delta \sum_{m\notin M_{0}}(1+\arrowvert x_m\arrowvert)^{-\beta}\lVert \chi_{\widetilde{Q_m}}\psi_k(\sqrt{L_{\alpha}})f\rVert_2^2\\
&\leq C2^{2(j_0-j)N}2^{-2\mathscr{l}N}\delta \sum_{m}\int_{\widetilde{Q_m}}\arrowvert \chi_{\widetilde{Q_m}}\psi_k(\sqrt{L_{\alpha}})f\arrowvert^2 (1+\arrowvert x\arrowvert)^{-\beta}d\mu_{\alpha}(x)\\
&\leq C2^{2(j_0-j)N}2^{-2\mathscr{l}N}\delta \int_{\mathbb{R}_+^d}\arrowvert \psi_k(\sqrt{L_{\alpha}})f\arrowvert^2 (1+\arrowvert x\arrowvert)^{-\beta}d\mu_{\alpha}(x).
\end{aligned}
\end{equation*}
Finally, recalling \eqref{B20} and taking sum over $\mathscr{l}$ yields the estimate \eqref{B25}.
\par Therefore, since $\delta\leq \delta A_{\beta,d}^{\varepsilon}(\delta)$, we thus deduce from \eqref{B21}, \eqref{B25} and \eqref{B25} that
\begin{equation}\label{B27}
\int_{2^{k-1}}^{2^{k+2}}E^{k,j}(t)\frac{dt}{t}\leq C_N2^{2(j_0-j)N} \delta A_{\beta,d}^{\varepsilon}(\delta)\int_{\mathbb{R}_+^d}\arrowvert \psi_k(\sqrt{L_{\alpha}})f\arrowvert^2 (1+\arrowvert x\arrowvert)^{-\beta}d\mu_{\alpha}(x)
\end{equation}
for any $N>0$ if $j>k $.
\\
\par \textbf{Case 2: $j\leq k$.}
In this case, the side length of each $Q_m$ is less than 4, then \eqref{B26} holds for any $m\in \mathbb{Z}_{\geq 0}^{d}$. Thus, the same argument in the proof of \eqref{B25} works without modification. We can deduce that
\begin{equation}\label{B28}
\int_{2^{k-1}}^{2^{k+2}}E^{k,j}(t)\frac{dt}{t}\leq C_N2^{2(j_0-j)N} \delta \int_{\mathbb{R}_+^d}\arrowvert \psi_k(\sqrt{L_{\alpha}})f\arrowvert^2 (1+\arrowvert x\arrowvert)^{-\beta}d\mu_{\alpha}(x)
\end{equation}
for any $N>0$.
\\
\\
\textit{Complection of the proof of \eqref{high}.}
\par By \eqref{B27}, \eqref{B28} and $\delta\leq \delta A_{\beta,d}^{\varepsilon}(\delta)$, we now have the estimate \eqref{B27} for any $j\geq j_0$ and $k$. Putting \eqref{B27} in the right hand side of \eqref{B9} and then taking sum over $j$, we obtain
\begin{equation}
\int_{\mathbb{R}_{+}^{d}}\arrowvert \mathfrak{S}_{\delta}^{l}f(x)\arrowvert ^2(1+\arrowvert x\arrowvert)^{-\beta}d\mu_{\alpha}(x)\leq C \delta A_{\beta,d}^{\varepsilon}(\delta)\int_{\mathbb{R}_+^d}\sum_{k}\arrowvert \psi_k(\sqrt{L_{\alpha}})f\arrowvert^2 (1+\arrowvert x\arrowvert)^{-\beta}d\mu_{\alpha}(x).
\end{equation}
Using the Littlewood-Paley inequality \eqref{PL1.1}, we get the estimate \eqref{Ahigh} and this completes the proof of Lemma \ref{SL}.
\end{appendix}
\section*{Data availability}
No data was used for the research described in the article.
\section*{Acknowledgments}
The author would like to express sincere gratitude to his supervisor, Professor Zhiwen Duan, for his invaluable advice and encouragement. The author also wishes to thank Hui Xu for her careful reading of the manuscript and for pointing out some printing errors, which have helped improve the clarity of the paper. 

\end{document}